\documentclass[11pt]{article}
\usepackage{geometry}
\usepackage{amssymb,amsmath}
 \usepackage{pstricks, pst-coil, pst-node, pst-tree, multido}
 \usepackage[english]{babel}
\usepackage{graphics}
\usepackage{graphicx}
 \newtheorem{DE}{Definition}[section]

\newcommand {\sm} {\setminus}
\usepackage{latexsym} %Des symboles de maths en plus.
\usepackage{theorem} %pour changer les styles dans les theoremes.
 \newcommand{\qed}{\relax\ifmmode\hskip2em\Box\else\unskip\nobreak\hfill$\Box$\fi}

\newtheorem{theorem}[DE]{Theorem}
\newtheorem{lemma}[DE]{Lemma}

{\theoremstyle{break}\theorembodyfont{\rmfamily}}
{\theoremstyle{break}\theorembodyfont{\rmfamily}}

\newcounter{claim}
\newenvironment{proof}[1][]%
	{\noindent {\setcounter{claim}{0}\sc proof --- }{#1}{}}{\qed\vspace{2ex}}
	{\refstepcounter{claim}\vspace{1ex}\noindent {(\it\arabic{claim}) {#1}{}}\it}{\vspace{1ex}}
	{\noindent {}{#1}{}}{ This proves~(\arabic{claim}).\vspace{1ex}}

\begin{document}

 \title{The (theta, wheel)-free graphs\\ Part III: cliques, stable sets and coloring}
\author{Marko Radovanovi\'c\thanks{University
    of Belgrade, Faculty of Mathematics, Belgrade, Serbia. Partially
    supported by Serbian Ministry of Education, Science and
    Technological Development project 174033. E-mail:
    markor@matf.bg.ac.rs}~, Nicolas Trotignon\thanks{CNRS, LIP, ENS de
    Lyon. Partially supported by ANR project Stint under reference
    ANR-13-BS02-0007 and by the LABEX MILYON (ANR-10-LABX-0070) of
    Universit\'e de Lyon, within the program ‘‘Investissements
    d'Avenir’’ (ANR-11-IDEX-0007) operated by the French National
    Research Agency (ANR).  Also Universit\'e Lyon~1, universit\'e de
    Lyon. E-mail: nicolas.trotignon@ens-lyon.fr}~, Kristina Vu\v
  skovi\'c\thanks{School of Computing, University of Leeds, and
    Faculty of Computer Science (RAF), Union University, Belgrade,
    Serbia.  Partially supported by EPSRC grant EP/N0196660/1, and
    Serbian Ministry of Education and Science projects 174033 and
    III44006. E-mail: k.vuskovic@leeds.ac.uk}}

\maketitle

\begin{abstract}
  A hole in a graph is a chordless cycle of length at least 4. A theta
  is a graph formed by three paths between the same pair of distinct
  vertices so that the union of any two of the paths induces a hole. A
  wheel is a graph formed by a hole and a vertex that has at least 3
  neighbors in the hole.
  In this series of papers we study the class of graphs that do not contain
  as an induced subgraph a theta nor a wheel. In Part II of the series we
  prove a decomposition theorem for  this class, that uses  clique
  cutsets and 2-joins, and consequently obtain a polynomial time recognition
  algorithm for the class. In this paper we further use this decomposition theorem
  to obtain polynomial time algorithms for maximum weight clique, maximum weight
  stable set and coloring problems. We also show that
  for a graph  $G$ in the class, if its maximum clique size is $\omega$, then its chromatic number is bounded by
  max$\{\omega ,3\}$, and that the class is 3-clique-colorable.
 \end{abstract}

 \section{Introduction}\label{sec:intro}

In this article, all graphs are finite and simple.
We say that a graph $G$   \emph{contains} a graph $H$ if
$H$ is isomorphic to an induced subgraph of $G$, and that $G$ is
\emph{$H$-free} if it does not contain $H$.  For a family of graphs ${\cal H}$,
$G$ is \emph{${\cal H}$-free} if for every $H\in {\cal H}$, $G$ is $H$-free.

A \emph{hole} in a graph is a chordless cycle of length at least 4.
A \emph{theta}
  is a graph formed by three paths between the same pair of distinct
  vertices so that the union of any two of the paths induces a hole. A
  \emph{wheel} is a graph formed by a hole and a vertex that has at least 3
  neighbors in the hole.

In this series of papers we study (theta, wheel)-free graphs.
This project is motivated and
explained in more detail in Part I of the
series~\cite{twf-p1}, where two subclasses of (theta, wheel)-free
graphs are studied.  In Part II of the series~\cite{twf-p2}, we prove
a decomposition theorem for (theta, wheel)-free graphs that uses clique cutsets
and 2-joins, and use it to obtain an $\mathcal O(n^4m)$-time recognition
algorithm for the class (where $n$ denotes the number of vertices and $m$ the number of edges of a given graph).
In this part we use the decomposition theorem from~\cite{twf-p2} to obtain further
properties of the graphs in the class
and to construct polynomial time algorithms for maximum weight clique,
maximum weight stable set, and coloring problems.
In Part IV of the series~\cite{twf-p4} we show that the induced version of
the $k$-linkage problem can be solved in polynomial time for (theta, wheel)-free
graphs.

\subsection*{The main results and the outline of the paper}

Throughout the paper we will denote by ${\cal C}$ the class of (theta, wheel)-free
graphs. Also, $n$ will denote the number of vertices and $m$ the number of edges of
a given graph.

For completeness, in Section~\ref{sec:decTh}, we state the decomposition
theorem for ${\cal C}$ and several other results proved in previous parts that
will be needed here.
Fundamental for our algorithms are the 2-join decomposition techniques developed in
 \cite{nicolas.kristina:two} which we also describe here, as well as
 prove some preliminary lemmas.

In Section~\ref{sec:maxCl}, we prove that every graph in ${\cal C}$
contains a bisimplicial vertex, and use this property to
give an $\mathcal O(n^2m)$-time
algorithm for the maximum weight clique problem on ${\cal C}$,
as well as to show that the class is 3-clique-colorable.

In Section~\ref{sec:stable}, we give an $\mathcal O(n^6m)$-time
algorithm for the maximum weight stable set problem on ${\cal C}$.

In Section~\ref{sec:vCol}, we give an $\mathcal O(n^5m)$-time
algorithm that optimally colors  graphs from ${\cal C}$.  We also
prove that every graph in ${\cal C}$, with maximum clique size
$\omega$, admits a coloring with at most $\max\{\omega, 3\}$ colors.

%In Section~\ref{sec:indCP}, we study the induced version of the
%$k$-linkage problem.  This problem is NP-hard in general, but we prove
%it is polynomial when restricted to (theta, wheel)-free graphs.

Since ${\cal C}$ contains all chordal graphs, clearly ${\cal C}$ has unbounded
clique-width.
In Section~\ref{s:cW}, we show how an example of Lozin and Rauthenbach
\cite{lozin:unboundedCW}
implies that the class of graphs from ${\cal C}$ that have no clique cutset
also has
unbounded clique-width.

\subsection*{Terminology and notation}

A {\em clique} in a graph is a (possibly empty) set of pairwise adjacent vertices.
%A clique on $k$ nodes is denoted by $K_k$.
We say that a clique is
{\it big} if it is of size at least 3.
%A $K_3$ is also referred to as
%a {\em triangle}, and is denoted by $\Delta$.
A {\em stable set} in a graph is a (possibly empty) set of pairwise nonadjacent vertices.
A {\it diamond} is a
graph obtained from a complete graph on 4 vertices by deleting an edge. A {\it claw} is a
graph induced by vertices $u,v_1,v_2,v_3$ and edges $uv_1,uv_2,uv_3$.

A {\em path} $P$ is a sequence of distinct vertices
$p_1p_2\ldots p_k$, $k\geq 1$, such that $p_ip_{i+1}$ is an edge for
all $1\leq i <k$.  Edges $p_ip_{i+1}$, for $1\leq i <k$, are called
the {\em edges of $P$}.  Vertices $p_1$ and $p_k$ are the {\em ends}
of $P$.  A cycle $C$ is a sequence of vertices $p_1p_2\ldots p_kp_1$,
$k \geq 3$, such that $p_1\ldots p_k$ is a path and $p_1p_k$ is an
edge.  Edges $p_ip_{i+1}$, for $1\leq i <k$, and edge $p_1p_k$ are
called the {\em edges of $C$}.  Let $Q$ be a path or a cycle.  The
vertex set of $Q$ is denoted by $V(Q)$.  The {\em length} of $Q$ is
the number of its edges.  An edge $e=uv$ is a {\em chord} of $Q$ if
$u,v\in V(Q)$, but $uv$ is not an edge of $Q$. A path or a cycle $Q$
in a graph $G$ is {\em chordless} if no edge of $G$ is a chord of
$Q$. % {\em Girth} of a
%graph is the length of its shortest cycle.

%Let $A$ and $B$ be two disjoint node sets such that no node of $A$ is adjacent
%to a node of $B$. A path $P=p_1, \ldots ,p_k$ {\em connects $A$ and $B$}
%if either $k=1$ and $p_1$ has neighbors in both $A$ and $B$, or
%$k>1$ and one of the two endnodes of $P$ is adjacent to at least one node
%in $A$ and the other endnode is adjacent to at least one node in $B$. The path
%$P$ is a {\em direct connection between $A$ and $B$} if in
%$G[V(P) \cup A \cup B]$ no path connecting $A$ and $B$ is shorter than $P$.
%The direct connection $P$ is said to be {\em from $A$ to $B$} if $p_1$ is
%adjacent to a node of $A$ and $p_k$ is adjacent to a node of $B$.

Let $G$ be  a graph.
For $x\in V(G)$, $N(x)$ is the set of all neighbors of $x$ in $G$, and $N[x]=N(x) \cup \{ x\}$.
%Let $H$ and $C$ be vertex-disjoint induced subgraphs of $G$.
%The {\em attachment of $C$ over $H$}, denoted by $N_H(C)$, is the set of all vertices of $H$ that have at least one
%neighbor in $C$.
%When $C$ consists of a single vertex $x$, we denote the attachment of $C$ over $H$ by $N_H(x)$,
%and we say that it is an {\em attachment of $x$ over $H$}. Note that $N_H(x)=N(x)\cap V(H)$.
%Also we say that $C$ is {\em $H$-complete}, if every vertex of $C$ is adjacent
%to every vertex of $H$.
For $S\subseteq V(G)$, $G[S]$ denotes the subgraph of $G$ induced by $S$.
For disjoint subsets $A$ and $B$ of $V(G)$, we say that $A$ is {\em complete} (resp. {\em anticomplete})
to $B$ if every vertex of $A$ is adjacent (resp. nonadjacent) to every vertex of $B$.

In a graph $G$, a
subset $S$ of vertices and/or edges is a {\em cutset} if its removal yields
a disconnected graph.

When clear from the context, we will sometimes write $G$ instead of $V(G)$.

\section{Decomposition of (theta, wheel)-free graphs}
\label{sec:decTh}

To state the decomposition theorem for graphs in ${\cal C}$ we first define the
basic classes involved and then the cutsets used.

\subsection*{Basic classes}

We will refer to P-graphs and line graphs of triangle-free chordless graphs (which we now define)
as {\em basic} graphs.

A graph $G$ is {\em chordless} if no cycle of $G$ has a chord.
% and it
%is {\em sparse} if for every edge $e=uv$, at least one of $u$ or $v$
%has degree at most~2.  Clearly all sparse graphs are chordless.
An
edge of a graph is {\em pendant} if at least one of its endnodes has
degree~1.  A \emph{branch vertex} in a graph is a vertex of degree at
least~3.  A {\em branch} in a graph $G$ is a path of length at least~1
whose internal vertices are of degree 2 in $G$ and whose endnodes are
both branch vertices.  A {\em limb} in a graph $G$ is a path of length
at least~1 whose internal vertices are of degree 2 in $G$ and whose
one endnode has degree at least 3 and the other one has degree~1. Two
distinct branches are {\em parallel} if they have the same endnodes.
Two distinct limbs are {\em parallel} if they share the same vertex of
degree at least~3.

Cut vertices of a graph $R$ that are also branch vertices are called
the {\em attaching vertices} of $R$.  Let $x$ be an attaching vertex
of a graph $R$, and let $C_1, \ldots ,C_t$ be the connected components
of $R\setminus x$ that together with $x$ are not limbs of $R$ (possibly, $t=0$, when all
connected components of $R\setminus x$ are limbs).  If $x$ is the end
of at least two parallel limbs of $R$, let $C_{t+1}$ be the subgraph of $R$ formed by
all the limbs of $R$ with endnode $x$.  The graphs
$R[V(C_i)\cup \{ x\} ]$ (for $i=1, \ldots, t$) and the graph $C_{t+1}$
(if it exists) are the \emph{$x$-petals} of $R$.

For any integer $k\geq 1$, a {\em $k$-skeleton}  is a graph $R$ such that:

\begin{enumerate}
\item $R$ is connected, triangle-free, chordless and contains at least
  three pendant edges (in particular, $R$ is not a path).

 \item $R$ has no parallel branches (but it may contain
   parallel limbs).

\item For every cut vertex $u$ of $R$, every component of
  $R\setminus u$ has a vertex of degree 1 in $R$.

\item For every vertex cutset $S=\{a, b\}$ of $R$ and for
  every component $C$ of $R\setminus S$, either $R[C\cup S]$ is a
  chordless path from $a$ to $b$, or $C$ contains at least one vertex
  of degree 1 in $R$.

\item For every edge $e$ of a cycle of $R$, at least one of the
  endnodes of $e$ is of degree 2.

\item Each pendant edge of $R$ is given one label, that is an integer
  from $\{1, \dots, k\}$.

\item Each label from $\{ 1, \ldots ,k\}$ is given at least once (as a
  label), and some label is used at least twice.

\item If some pendant edge whose one endnode is of degree at least 3
  receives label $i$, then no other pendant edge receives label $i$.

\item If $R$ has no branches then $k=1$, and otherwise
 if two limbs of $R$ are parallel, then
  their pendant edges receive different labels and at least one of these labels is used more then once.

\item If $k>1$ then for every attaching vertex $x$ and for
  every $x$-petal $H$ of $R$, there are at least two distinct labels
  that are used in $H$.  Moreover, if $\overline{H}$ is a union of at least one but not all $x$-petals,
  then there is a label $i$ such that both $\overline{H}$ and $(R\setminus\overline{H})\cup\{x\}$ have
  pendant edges with label $i$.

\item If $k=2$, then both labels are used at least twice.
\end{enumerate}

Note that if  $R$ is a skeleton, then it edgewise
partitions into its branches and its limbs. Also, there is a trivial
one-to-one correspondence between the pendant edges of $R$ and the
limbs of $R$: any pendant edge belongs to a unique limb, and
conversely any limb contains a unique pendant edge.

If $R$ is a graph, then the {\em line graph} of $R$, denoted by $L(R)$, is the graph whose vertices are the edges of $R$,
and such that two vertices of $L(R)$ are adjacent if and only if the corresponding edges are adjacent in $R$.

A {\em P-graph}  is any graph $B$ that can be
constructed as follows:

\begin{itemize}
\item
Pick an integer $k\geq 1$ and a $k$-skeleton $R$.
\item
Build $L(R)$, the line graph of $R$. The vertices of $L(R)$ that
correspond to pendant edges of $R$ are called {\em pendant vertices}
of $L(R)$, and they receive the same label as their corresponding
pendant edges in $R$.
\item Build a clique $K$ with vertex set $\{ v_1, \ldots ,v_k\}$,
  disjoint from $L(R)$.
\item $B$ is now constructed from $L(R)$ and $K$ by adding edges
  between $v_i$ and all pendant vertices of $L(R)$ that have label
  $i$, for $i=1, \ldots ,k$.
\end{itemize}

We say that $K$ is the {\em special clique}
of $B$ and $R$ is the {\em skeleton} of $B$.

\begin{lemma}
  \label{l:twoBranches}
  Every P-graph $G$ contains two distinct branches of length at least~2
  (in particular, these two branches both contain a vertex of degree~2).
\end{lemma}

\begin{proof}
  Let $i$ be a label of $G$ that is used at least twice (it exists by (vii)) and consider two
  pendant edges of the skeleton $R$ of $G$ that receive this label. Then,
   by condition (viii) the limbs that contain these pendant edges are of
   length at least 2, and hence they correspond to branches of length at
   least 2 in $G$ (note that by (i) the degree of $v_i$ in $G$ is at least 3).
\end{proof}

\begin{lemma}[\cite{twf-p1}]\label{p1l2.4}
$G$ is the line graph of a triangle-free chordless graph if and only if $G$ is (wheel, diamond, claw)-free.
\end{lemma}

\begin{lemma}[\cite{twf-p2}]\label{p2l4.2}
Every P-graph is (theta, wheel, diamond)-free.
\end{lemma}

\subsection*{Cutsets}

 A vertex cutset $S$ is a {\em clique cutset} if
$S$ is a clique.  Note that every disconnected graph has a clique
cutset: the empty set.

An {\em almost 2-join} in a graph $G$ is a pair $(X_1,X_2)$ that is a
partition of $V(G)$, and such that:

\begin{itemize}
\item For $i=1,2$, $X_i$ contains disjoint nonempty sets $A_i$ and
  $B_i$, such that every vertex of $A_1$ is adjacent to every vertex   of $A_2$, every vertex of $B_1$ is adjacent to every vertex of
  $B_2$, and there are no other adjacencies between $X_1$ and $X_2$.
\item For $i=1,2$, $|X_i|\geq 3$.
\end{itemize}

An almost 2-join $(X_1, X_2)$ is a \emph{2-join} when for $i\in\{1,2\}$, $X_i$
contains at least one path from $A_i$ to $B_i$, and if $|A_i|=|B_i|=1$
then $G[X_i]$ is not a chordless path.

We say that $(X_1,X_2,A_1,A_2,B_1,B_2)$ is a {\em split} of this
2-join, and the sets $A_1,A_2,B_1,B_2$ are the {\em special sets} of
this 2-join.  We often use the following notation:
$C_i = X_i\sm (A_i \cup B_i)$ (possibly, $C_i = \emptyset$).

We are ready to state the decomposition theorem from~\cite{twf-p2}.

\begin{theorem}[\cite{twf-p2}]\label{decomposeTW}
  If $G$ is (theta, wheel)-free, then $G$ is a line graph of a
  triangle-free chordless graph or a P-graph, or $G$ has a clique
  cutset or a 2-join.
\end{theorem}

We now describe how we decompose a graph from ${\cal C}$ into basic
graphs using the cutsets in the above theorem.

\subsection*{Decomposing with clique cutsets}

If a graph $G$ has a clique cutset $K$, then
its vertex set can be partitioned into sets $(A,K,B)$, where $A$ and $B$
are nonempty and   anticomplete.  We say that $(A,K,B)$
is a \emph{split} for the clique cutset $K$.  When $(A, K, B)$ is a
split for a clique cutset of a graph $G$, the {\em blocks of decomposition}
of $G$ with respect to $(A, K, B)$ are the graphs $G_A=
G[A\cup K]$ and $G_B= G[K \cup B]$.

A \emph{clique cutset decomposition tree} for a graph $G$ is a rooted
tree $T$ defined as follows.

\begin{itemize}
\item[(i)] The root of $T$ is $G$.
\item[(ii)] Every non-leaf vertex of $T$ is a graph $G'$ that contains a
  clique cutset $K'$ with split $(A', K', B')$. The children of $G'$
  in $T$ are the blocks of decomposition $G'_{A'}$ and $G'_{B'}$ of
  $G'$ with respect to $(A', K', B')$, and at least one of the graphs
  $G_{A'}'$ and $G_{B'}'$ do not admit a clique cutset.
  \item[(iii)] Every leaf of $T$ is a graph with no clique cutset.
  \item[(iv)] $T$ has at most $n$ leaves.
  \end{itemize}

\begin{theorem}[\cite{tarjan}]
  \label{th:tarjan}
  A clique cutset decomposition tree of an input graph $G$ can be
  computed in time $O(nm)$.
\end{theorem}

Note that for a non-leaf vertex $G'$ of $T$, the corresponding
clique cutset $K'$ of $G'$ is also a clique cutset of $G$.
The following lemmas proved in \cite{twf-p1} will also be needed.

\begin{lemma}[\cite{twf-p1}]\label{diamondCliqueCut}
  If $G$ is a wheel-free graph that contains a diamond, then $G$ has a
  clique cutset.
\end{lemma}

A {\em star cutset} in a graph is a vertex cutset $S$ that contains a
vertex (called a {\em center}) adjacent to all other vertices of $S$. Note
that a nonempty clique cutset is a star cutset.

\begin{lemma}[\cite{twf-p1}]\label{Star=Clique}
  If $G\in\mathcal C$ has a star cutset, then $G$ has a clique cutset.
\end{lemma}

\subsection*{Decomposing with 2-joins}

We first state some properties of 2-joins in graphs with no clique cutset.
Let $\mathcal D$ be the class of all graphs from $\mathcal C$ that do
not have a clique cutset. By Lemma~\ref{Star=Clique}, no graph from
$\mathcal D$ has a star cutset and by Lemma \ref{diamondCliqueCut} no
graph from $\mathcal D$ contains a diamond. Also, let
$\mathcal D_{\textsc{basic}}$ be the class of all basic graphs from
$\mathcal C$ that do not have a clique cutset.

An almost 2-join with a split $(X_1, X_2, A_1, A_2, B_1, B_2)$ in a
graph $G$ is \emph{consistent} if the following statements hold for
$i=1, 2$:

\begin{enumerate}
\item Every component of $G[X_i]$ meets both $A_i$, $B_i$.
\item Every vertex of $A_i$ has a non-neighbor in $B_i$.
\item Every vertex of $B_i$ has a non-neighbor in $A_i$.
\item Either both $A_1$, $A_2$ are cliques, or one of $A_1$ or $A_2$ is
  a single vertex, and the other one is a disjoint union of cliques.
\item Either both $B_1$, $B_2$ are cliques, or one of $B_1$, $B_2$ is
  a single vertex, and the other one is a disjoint union of cliques.
\item $G[X_i]$ is connected.
\item For every vertex $v$  in $X_i$, there exists a path in $G[X_i]$
  from $v$ to some vertex of $B_i$ with no internal vertex in $A_i$.
\item For every vertex $v$  in $X_i$, there exists a path in $G[X_i]$
  from $v$ to some vertex of $A_i$ with no internal vertex in $B_i$.
\end{enumerate}

Note that the definition contains redundant statements (for instance, (vi)
implies (i)), but it is convenient to list properties separately as above.

\begin{lemma}[\cite{twf-p1}]
  \label{l:consistent}
  If $G\in\mathcal D$, then
  every almost 2-join of $G$ is consistent.
\end{lemma}

By this lemma every 2-join of a graph of $\mathcal D$ is consistent.

We now define the blocks of decomposition of a graph with respect to a
2-join.  Let $G$ be a graph and $(X_1, X_2,A_1,A_2,B_1,B_2)$ a split of a 2-join of $G$.
Let $k_1$ and $k_2$ be positive integers. The
\emph{blocks of decomposition} of $G$ with respect to $(X_1, X_2)$ are
the two graphs $G_1^{k_1}$ and $G_2^{k_2}$ that we describe now.  We obtain $G_1^{k_1}$
from $G$ by replacing $X_2$ by a \emph{marker path} $P_2= a_2 \ldots
b_2$ of length $k_1$, where $a_2$ is a vertex complete to $A_1$, $b_2$ is a vertex complete
to $B_1$, and  $V(P_2)\setminus \{ a_2,b_2\}$ is anticomplete to $X_1$.  The
block $G_2^{k_2}$ is obtained similarly by replacing $X_1$ by a marker path
$P_1 = a_1\ldots b_1$ of length $k_2$.

In \cite{twf-p2} the blocks of decomposition w.r.t.\ a 2-join that we used in construction of a
recognition algorithm had marker paths of length 2. In this paper we will use blocks whose
marker paths are of length 3. So, unless otherwise stated, when we say that $G_1$ and
$G_2$ are blocks of decomposition w.r.t.\ a 2-join we will mean that their marker paths are of length~3.

\begin{lemma}[\cite{twf-p1}]
  \label{l:keepKfree}
  Let $G$ be a graph with a consistent 2-join $(X_1, X_2)$ and $G_1$,
  $G_2$ be the blocks of decomposition with respect to this 2-join whose marker
  paths are of length 2. Then the following hold:
  \begin{itemize}
 \item[(i)] $G$ has no clique cutset if and only if $G_1$ and $G_2$ have
  no clique cutset.
  \item[(ii)] $G\in {\cal C}$ if and only if $G_1$ and $G_2$ are in ${\cal C}$.
  \end{itemize}
\end{lemma}

\begin{lemma}
  \label{new2}
  Let $G$ be a graph from ${\cal D}$. Let $(X_1, X_2)$ be a 2-join of $G$, and $G_1$,
  $G_2$ the blocks of decomposition with respect to this 2-join whose marker
  paths are of length at least~2. Then
  $G_1$ and $G_2$ are in ${\cal D}$ and they do not have star cutsets.
\end{lemma}

\begin{proof}
By Lemma \ref{l:consistent}, $(X_1,X_2)$ is consistent.
Let $G_1'$ and $G_2'$ be blocks of decomposition w.r.t.\ $(X_1,X_2)$ whose marker paths are of length 2.
Then for $i\in \{ 1,2\}$, $G_i$ is obtained from $G_i'$ by subdividing (0 or several times) an edge of its marker path.
Subdividing an edge whose one endnode is of degree 2 cannot create a clique cutset, nor a theta, nor a wheel,
and hence the result follows from Lemma \ref{l:keepKfree} and Lemma \ref{Star=Clique}.
\end{proof}

A 2-join $(X_1,X_2)$ of $G$ is a {\em minimally-sided 2-join} if for some $i\in \{ 1,2 \}$ the following holds: for every 2-join $(X_1',X_2')$ of $G$, neither $X_1'\subsetneq X_i$ nor $X_2'\subsetneq X_i$. In this case $X_i$ is a {\em minimal side} of this minimally-sided 2-join.

A 2-join $(X_1,X_2)$ of $G$ is an {\em extreme 2-join} if for some $i\in \{ 1,2 \}$ and all $k\geq 3$ the block of decomposition $G_i^k$ has no 2-join.
In this case $X_i$ is an {\em extreme side} of such a 2-join.

Graphs in general do not necessarily have extreme 2-joins (an example is given in \cite{nicolas.kristina:two}), but it is shown in  \cite{nicolas.kristina:two}
that graphs with no star cutset do. It is also shown in \cite{nicolas.kristina:two} that if $G$ has no star cutset then the blocks of decomposition w.r.t.\ a 2-join
whose marker paths are of length at least 3, also have no star cutset. This is then used to show that in a graph with no star cutset, a minimally-sided 2-join is extreme.
We summarize these results in the following lemma.

\begin{lemma}[\cite{nicolas.kristina:two}]\label{extreme}
Let $G$ be a graph with no star cutset. Let
$(X_1,X_2,A_1,A_2,B_1,B_2)$ be a split of a minimally-sided 2-join of
$G$ with $X_1$ being a minimal side, and let $G_1$ and $G_2$ be the
corresponding blocks of decomposition whose marker paths are of length at least 3. Then the following hold:
\begin{enumerate}
\item $|A_1|\geq 2$, $|B_1|\geq 2$, and in particular all the vertices of $A_2\cup B_2$ are of degree at least~3.
\item If $G_1$ and $G_2$ do not have star cutsets, then $(X_1,X_2)$ is an extreme 2-join, with $X_1$ being an extreme side (in particular, $G_1$ has no 2-join).
\end{enumerate}
\end{lemma}

The following simple lemma is useful and not proved in the previous papers of
the series.

\begin{lemma}
  \label{l:2joinDeg2}
  Let $G$ be in $\mathcal D$.  Let $(X_1,X_2,A_1,A_2,B_1,B_2)$ be a
  split of a minimally-sided 2-join of $G$ with $X_1$ being a minimal
  side, and let $G_1$ and $G_2$ be the corresponding blocks of
  decomposition. If the block of decomposition $G_1$ is a P-graph,
  then $X_1$ contains a vertex that has degree~2 in $G$.
\end{lemma}

\begin{proof}
  By Lemma~\ref{l:twoBranches}, $G_1$ contains a  vertex $v$ of
  degree~2 that is not in the marker path of $G_1$.  We claim that $v$
  has also degree~2 in $G$. If $v\in X_1\setminus(A_1\cup B_1)$, then
  it is clear, so suppose $v$ is in $A_1\cup B_1$, say in $A_1$ up to
  symmetry.  Note that $(X_1, X_2)$ is consistent by
  Lemma~\ref{l:consistent}.  Since $v$ has degree~2 in $G_1$,
  condition (vii) in the definition of consistent 2-joins applied to
  $v$ implies that $v$ has precisely one neighbor in
  $X_1\setminus A_1$ and one neighbor in the marker path of $G_1$.
  Since by Lemma~\ref{extreme} $|A_1|\geq 2$, it follows that $G[A_1]$
  is disconnected.  Hence, by condition (iv) in the definition of
  consistent 2-joins, $|A_2|=1$. It follows that $v$ has the same
  degree in $G_1$ and in $G$.
\end{proof}

In \cite{nicolas.kristina:two} it is shown that one can decompose a graph with no star cutset
using a sequence of `non-crossing' 2-joins into graphs with no star cutset and no 2-join
(which will in our case be basic). This will be particularly important when using 2-join decomposition to
solve the stable set problem. We now describe such 2-join decomposition obtained in \cite{nicolas.kristina:two}.

A {\it flat path} of $G$ is any path of $G$ of length at least 3, whose
interior vertices are of degree 2, and whose ends do not have a common neighbor.
When $\mathcal M$ is a collection of vertex-disjoint flat paths of
$G$, a 2-join $(X_1,X_2)$ of $G$ is $\mathcal M$-independent if for
every path $P$ from $\mathcal M$ we have that either $V(P)\subseteq X_1$ or
$V(P)\subseteq X_2$.

\subsubsection*{2-Join decomposition tree $T_G$ of depth $p\geq 1$ of a graph $G$ that has no star cutset and has a 2-join}
\begin{itemize}
\item[(i)] The root of $T_G$ is $(G^0,\mathcal M^0)$, where $G^0:=G$ and $\mathcal M^0=\emptyset$.
\item[(ii)] Each vertex of $T_G$ is a pair $(H,\mathcal M)$, where $H$ is a graph of $\mathcal D$ and $\mathcal M$ is a set of disjoint flat paths of $H$.

    The non-leaf vertices of $T_G$ are pairs  $(G^0,\mathcal M^0),\ldots,(G^{p-1},\mathcal M^{p-1})$. Each non-leaf vertex $(G^i,\mathcal M^i)$ has two children. One is $(G^{i+1},\mathcal M^{i+1})$, the other one is $(G_B^{i+1},\mathcal M_B^{i+1})$.

    The leaf-vertices of $T_G$ are the pairs $(G_B^1,\mathcal M_B^1),\ldots,(G_B^{p},\mathcal M_B^{p})$ and $(G^p,\mathcal M^p)$.
    Graphs $G_B^1,G_B^2,\ldots,G_B^p,G^p$ have no star cutset nor  2-join.

  \item[(iii)] For $i\in\{0,1,\ldots,p-1\}$, $G^i$ has a 2-join $(X_1^i,X_2^i)$ that is extreme with extreme side $X_1^i$ and that is $\mathcal M^i$-independent.
  Graphs $G^{i+1}$ and $G_B^{i+1}$ are blocks of decomposition of $G^i$ w.r.t.\ $(X_1^i,X_2^i)$ whose marker paths are of length at least 3.
  The block $G_B^{i+1}$ corresponds to the extreme side $X_1^i$, i.e.\ $X_1^i\subseteq V(G_B^{i+1})$.

      Set $\mathcal M_B^{i+1}$ consists of paths from $\mathcal M^i$ whose vertices are in $X_1^i$. Note that the marker path used to construct the block $G_B^{i+1}$ does not belong to $\mathcal M_B^{i+1}$.

       Set $\mathcal M^{i+1}$ consists of paths from $\mathcal M^i$ whose vertices are in $X_2^i$ together with the marker path $P^{i+1}$ used to build $G^{i+1}$.

\item[(iv)] $\mathcal M_B^1\cup\ldots\cup \mathcal M_B^p\cup \mathcal M^{p}$ is the set of all marker paths used in the construction of the vertices $G^1,\ldots,G^p$ of $T_G$, and the sets $\mathcal M_B^1,\ldots,\mathcal M_B^p,\mathcal M^p$ are pairwise disjoint.
  \end{itemize}

Vertex $(G^p,\mathcal M^p)$ is a leaf of $T_G$ and is called the {\it deepest vertex of $T_G$}.

The 2-join decomposition tree is described slightly differently in \cite{nicolas.kristina:two}, but the following result follows easily from the proofs in \cite{nicolas.kristina:two}.

\begin{lemma}[\cite{nicolas.kristina:two}]\label{DT-construction}
There is an algorithm with the following specification.
\begin{description}
\item[ Input:] A graph $G$ that has no star cutset and has a 2-join.
\item[ Output:]
A 2-join decomposition tree $T_G$ of depth at most $n$. %whose leaves have no star cutset nor 2-join.
\item[ Running time:] $\mathcal O(n^4m)$.
\end{description}
\end{lemma}

\begin{lemma}\label{new3}
If $G\in {\cal D}$ has a 2-join, then $T_G$ can be constructed, and all graphs
$G_B^1,G_B^2,\ldots,G_B^p,G^p$ that correspond to the leaves of $T_G$ are in $\mathcal D_{\textsc{basic}}$.
\end{lemma}

\begin{proof}
By Lemma \ref{Star=Clique} $G$ has no star cutset, and hence we can construct $T_G$.
By Lemma \ref{new2} all graphs that correspond to vertices of $T_G$ belong to ${\cal D}$.
By construction graphs $G_B^1,G_B^2,\ldots,G_B^p,G^p$ have no star cutset nor 2-join.
By Lemma \ref{Star=Clique} it follows that none of them has a clique cutset, and hence by Theorem \ref{decomposeTW}
all of them are basic.
\end{proof}

\section{Maximal cliques and clique coloring}
\label{sec:maxCl}

A vertex $v$ of a graph $G$ is {\it simplicial} if $N(v)$ is a
clique, and  it is {\it bisimplicial} if $N(v)$ is a disjoint
union of two cliques that are anticomplete to each other.
Note that every simplicial vertex is also bisimplicial.
We now show that every graph $G\in\mathcal C$ has a bisimpicial vertex, which we then use
 to obtain an algorithm for finding a maximum weight clique of $G$ and to prove that $G$ is 3-clique-colorable.

\begin{theorem}\label{bisimplicial}
  If $G\in\mathcal C$ then for every clique $K$
  of $G$, either $K= V(G)$ or there is  a bisimplicial vertex
  (of $G$) in $G\setminus K$.
\end{theorem}

\begin{proof}
The proof is by induction on $|V(G)|$.

If $R$ is a triangle-free chordless graph, then $L(R)$ does not
contain a claw nor a diamond, and hence every vertex of $L(R)$ is
 bisimplicial. If $G$ is a P-graph, then by
Lemma~\ref{l:twoBranches} it contains at least two branches of length
at least~2.  The clique $K$ contains internal vertices of at most one
of these branches. Hence, $G\setminus K$ contains a vertex of
degree~2, that is therefore bisimplicial. So, when $G$
is basic the result holds.

Let us now suppose that $(A,K',B)$ is a split of a clique cutset $K'$
of $G$, and let $G_A$ and $G_B$ be the blocks of decomposition w.r.t.\
this clique cutset. Then clique $K$ is contained in $G_A$ or in
$G_B$. W.l.o.g.\ suppose that $K$ is contained in $G_B$. By induction,
there is a bisimplicial vertex in $G_A\setminus K'$, and
hence in $G\setminus K$.

So, let us suppose that $G$ is not basic and that it does not admit a
clique cutset.  By Theorem \ref{decomposeTW}, $G$ admits a 2-join
$(X_1',X_2',A_1',A_2',B_1',B_2')$. Then $K$ is contained in
$G[X_1'\cup A_2']$ or in $G[X_2'\cup B_1']$. W.l.o.g.\ suppose that
$K$ is contained in $G[X_2'\cup B_1']$. Let
$(X_1,X_2,A_1,A_2,B_1,B_2)$ be a split of a minimally-sided 2-join of
$G$ with $X_1\subseteq X_1'$ being a minimal side, and let $G_1$ and
$G_2$ be the corresponding blocks of decomposition.
By Lemma
\ref{Star=Clique}, $G$ does not have a star cutset.
So by Lemma \ref{extreme} (ii) $G_1$ does not have a 2-join.
By Lemma \ref{new2}, $G_1\in {\cal D}$, and so
by Theorem \ref{decomposeTW}, $G_1$ is basic.
Additionally, by Lemma \ref{extreme}~(i), $|A_1|,|B_1|\geq 2$, and
hence, by (iv) and (v) of definition of consistent 2-join, $A_2$ and $B_2$ are cliques. Also
we may assume that
$K\cap A_1=\emptyset$, since otherwise for $u\in K\cap A_1$ and $v\in K\cap B_1$,
$u,v\in B_1'$, and hence any $b\in B_2'$ is a vertex of $X_2'$ that has a neighbor in both
$A_1$ and $B_1$, contradicting the assumption that $(X_1,X_2)$ is a 2-join of $G$
such that $X_1\subseteq X_1'$.
It follows that $K\subseteq X_2\cup B_1$.
If $G_1$ is the line graph
of a triangle-free chordless graph, then every vertex of $A_1$ is
 bisimplicial in $G_1$ and hence  bisimplicial
in $G$ (since $A_2$ is a clique). So, let us assume that $G_1$ is a P-graph. By
Lemma~\ref{l:2joinDeg2},  a vertex $u$ of $X_1$ is  of degree 2 in $G$.
If $B_1$ is a clique then, since $|B_1|\geq 2$ and by (viii) of definition of consistent 2-join, it follows that $u\not\in B_1$, and therefore $u\not\in K$ and the result holds (since $A_2$ is a clique).
If $B_1$ is not a clique, then $b_2$ is a center of a claw of $G_1$, and hence it is contained in the special clique $K'$ of P-graph $G_1$ (where $b_2$ is the vertex of the marker path of $G_1$ that is complete to $B_1$). So $K'\subseteq B_1\cup \{b_2\}$. Now, since every vertex of $X_1\setminus K'$ is bisimplicial in $G_1$, every vertex of $X_1\setminus B_1$ is bisimplicial in $G$ (since $A_2$ is a clique).
\end{proof}

\subsection*{Maximum weight clique}

Let $G$ be a graph and  $w:V(G)\rightarrow [0,+\infty)$ a {\it weight function on $G$}. A {\it maximum weight clique} of $G$
is a clique $K$ of $G$, such that $\sum_{v\in K} w(v)$ has the maximum value. If $K$ is a maximum weight clique of $G$,
we denote by $\omega_w(G)$ the value of the sum  $\sum_{v\in K} w(v)$.

\begin{theorem}\label{SMaxCliqueAlg}
There is an algorithm with the following specifications:
\begin{description}
\item[ Input:] A weighted graph $G\in\mathcal C$.
\item[ Output:]
A maximum weight clique of $G$.
\item[ Running time:] $\mathcal O(n^2m)$.
\end{description}
\end{theorem}

\begin{proof}
By Theorem \ref{bisimplicial}, $G$ contains a vertex $v$ that is bisimplicial. This vertex can be found in time $\mathcal O(nm)$. Let $N(v)$ consist of
(possibly empty) cliques $K_1$ and $K_2$.
Then  \[\omega_w(G)=\max\{\omega_w(G\setminus v),\omega_w(\{v\}\cup K_1),\omega_w(\{v\}\cup K_2)\},\] and if $K$ is the maximum weight clique of
$G\setminus v$, then a maximum weight clique of $G$ is $K$, $\{v\}\cup K_1$ or $\{v\}\cup K_2$.
So it is enough to find a maximum weight clique of $G\setminus v$, which can be done by applying (recursively) the same procedure on $G\setminus v$.

The total running time of this algorithm is $\mathcal O(n\cdot nm)=\mathcal O(n^2m)$.
\end{proof}

\subsection*{Clique coloring}

A {\it k-clique-coloring} of a graph $G$ is a function $c:V(G)\rightarrow \{1,2,\ldots,k\}$, such that for every inclusion-wise maximal clique $K$ of size at least 2,
$c(K)=\{c(v)\,:\,v\in K\}$ has at least 2 elements. We say  that $G$ is {\it $k$-clique-colorable} if it admits a $k$-clique-coloring. The {\it clique-chromatic number} of $G$, denoted by $\chi_C(G)$, is the smallest number $k$ such that $G$ is $k$-clique-colorable.

There are graphs in $\mathcal C$ that are not 2-clique-colorable, as
shown in Figure \ref{fig:3-clique-color}.
We now prove that every $G\in\mathcal C$ is $3$-clique-colorable.

 \begin{figure}[h!]
   \begin{center}
     \psset{xunit=21.0mm,yunit=21.0mm,radius=0.1,labelsep=0.5mm}
     \def\rputnode(#1,#2)#3#4{\Cnode(#1,#2){#3}\nput{0}{#3}{\small$#4$}}
     \begin{pspicture}(3,2.4)

       \rputnode(1,1){a}{}
       \rputnode(2,1){b}{}
       \rputnode(1,0){a1}{}
       \rputnode(2,0){b1}{}
       \rputnode(1.5,1.7){c}{}

       \rputnode(0.3,1.5){a2}{}
       \rputnode(2.7,1.5){b2}{}
       \rputnode(0.8,2.2){c1}{}
       \rputnode(2.2,2.2){c2}{}

       \ncline{a}{b}
       \ncline{a}{c}
       \ncline{b}{c}
       \ncline{a}{a1}
       \ncline{a}{a2}
       \ncline{b}{b1}

       \ncline{b}{b2}
       \ncline{c}{c1}
       \ncline{c}{c2}
       \ncline{a1}{b1}

       \ncline{a2}{c1}
       \ncline{b2}{c2}
       \ncline{0x}{0x3}

     \end{pspicture}
   \end{center}
 \caption{\label{fig:3-clique-color} Graph from $\mathcal C$ that is not 2-clique-colorable}
 \end{figure}
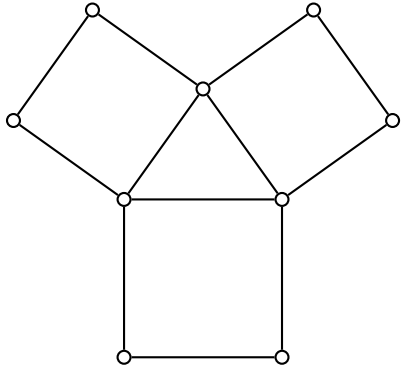

\begin{theorem}
If $G\in\mathcal C$, then $\chi_C(G)\leq 3$.
\end{theorem}

\begin{proof}
The proof is by induction on $|V(G)|$.
By Theorem \ref{bisimplicial}, $G$ contains a vertex $v$  that is bisimplicial.
By induction, we can 3-clique-color $G\setminus v$.
Let $K_1$ and $K_2$ be disjoint, anticomplete cliques  such that $K_1\cup K_2=N(v)$.
To obtain a 3-clique-coloring of $G$ from the 3-clique-coloring of $G\setminus v$ it is enough to color $v$
with a color different from a vertex of $K_1$ and a vertex of $K_2$ (note that if $K_i$ is empty, for some $i\in \{ 1,2\}$, then any of the three
colors satisfies the property).
\end{proof}

\section{Stable set problem}
\label{sec:stable}

Let $G$ be a graph and  $w:V(G)\rightarrow [0,+\infty)$ a {\it weight function on $G$}. A {\it maximum weight stable set} of $G$ is a stable set $S$ of
$G$, such that $\sum_{v\in S} w(v)$ has the maximum value. If $S$ is a maximum weight stable set of $G$, we denote by
$\alpha_w(G)$ the value of the sum  $\sum_{v\in S} w(v)$.

In this section we give a polynomial-time algorithm for finding a maximum weight stable set of a weighted graph in $\mathcal C$.
To do this we first introduce a different way to decompose w.r.t.\ a 2-join, one that is suited for the stable set problem.

%\subsection*{The gem block}

A {\it gem} $\Gamma$ is the graph defined with $V(\Gamma)=\{p_1,p_2,p_3,p_4,z\}$ and  $E(\Gamma)=\{p_1p_2,p_2p_3,p_3p_4,p_1z,p_2z,p_3,p_4z\}$. Vertex $z$ is the {\it center} of the gem $\Gamma$.
Let $(X_1,X_2,A_1,A_2,B_1,B_2)$ be a split of a  2-join of $G\in\mathcal C$ and $C_i=X_i\setminus(A_i\cup B_i)$, for $i\in\{1,2\}$. To build a {\it gem-block} $G_2^g$ replace $X_1$ by an induced path $pxyq$
plus a vertex $z$ complete to this path, such that $p$ (resp.\ $q$) is complete to $A_2$ (resp.\ $B_2$)
and these are the only edges between $\{ p,x,y,q,z\}$ and $X_2$.
Note that $G_2^g$ is not necessarily in $\mathcal C$.
Let $a:=\alpha_w(G[A_1\cup C_1])$, $b:=\alpha_w(G[B_1\cup C_1])$, $c:=\alpha_w(G[C_1])$ and $d:=\alpha_w(G[X_1])$. We give the following weights to the new vertices of $G_2^g$: $w(p)=a$, $w(x)=a+b-d$, $w(y)=d$, $w(q)=2d-a$ and $w(z)=c+d$.

\begin{lemma}[\cite{nicolas.kristina:two}]\label{gem}
If $G_2^g$ is the gem-block of $G$, then the weights of $G_2^g$ are non-negative and
$\alpha_w(G_2^g)=\alpha_w(G)+d$.
\end{lemma}

The gem blocks are useful for computing $\alpha$, but they are not preserving for our class ${\cal C}$, so
we cannot recursively decompose using gem blocks.
Instead, we will first construct the 2-join decomposition tree  $T_G$ (as in Section \ref{sec:decTh}) using marker paths of length 3,
and then we will reprocess it by replacing marker paths by gems.
As a consequence, the leaves of our decomposition tree may fail to be basic. So, we define {\it extensions} of basic graphs in the following way.

Let $P$ be a flat path of $G$ of length 3. {\it Extending} $P$ means adding a new vertex $z$
that is complete to $V(P)$ and anticomplete to the rest of the graph.
An {\it extension} of a pair $(G,\mathcal M)$, where $G$ is a graph and $\mathcal M$ a set of vertex-disjoint flat paths of $G$ of length 3,
 is any weighted graph obtained by extending the flat paths of $\mathcal M$ and giving any non-negative weights to all the vertices.
 An {\it extension} of $G$ is any graph that is an extension of $(G,\mathcal M)$ for some $\mathcal M$.
We define $\mathcal D_{\textsc{basic}}^{\textsc{ext}}$ to be the class of all graphs that are an extension of a graph from $\mathcal D_{\textsc{basic}}$.

 Let us examine the graphs in $\mathcal D_{\textsc{basic}}^{\textsc{ext}}$.
If $G$ is a line graph of a triangle-free chordless graph, then an extension $G'$ of $G$
is again a line graph (but not of a triangle-free chordless graph). Indeed, if $R$ is the root graph of $G$, then a flat path of $G$ of length 3 correspond to a path
$B=b_1b_2b_3b_4b_5$ of $R$ all of whose interior vertices are of degree 2. Hence, $G'=L(R')$, where $R'$ is the graph obtained from $R$ by adding edge $b_2b_4$
for every such $B$.
Similarly, if $G$ is a P-graph with special clique $K$ and skeleton $R$, then each flat path $P=a_1a_2a_3a_4$ of $G$ either corresponds  to a path $B=b_1b_2b_3b_4b_5$ of $R$
that belongs to a branch or a limb of $R$, or an endnode of $P$, say $a_4$ belongs to $K$
(in the latter case let $C=c_1c_2c_3c_4$ be the subpath of a limb of $R$ such that $L(R[V(C)])$ is the path $a_1a_2a_3$ in $G$).
To obtain $R'$ from $R$, for each $B$ we add the edge $b_2b_4$, and for each $C$ we add the edge $c_2c_4$.
Then an extension $G'$ of $G$ is obtained from $L(R')$ by adding clique $K$ and edges between them so that
for every $v\in K$, $N_{G'}(v)=N_G(v)\cup Z_v$, where $Z_v$ is the set of all centers of gems that were used to extend flat paths with endnode  $v$.

\begin{lemma}\label{SSbasicAlg}
There is an algorithm with the following specifications:
\begin{description}
\item[ Input:] A weighted graph $G'\in\mathcal D_{\textsc{basic}}^{\textsc{ext}}$.
\item[ Output:]
A maximum weight stable set of $G'$.
\item[ Running time:] $\mathcal O(n^4)$.
\end{description}
\end{lemma}

\begin{proof}
Let $G'$ be an extension of $G\in \mathcal D_{\textsc{basic}}$. In order to compute a maximum weight stable set of $G'$, we first need to compute $G$ and then decide if $G$ a line graph of a triangle-free chordless graph or a P-graph.
Since $G$ is diamond-free, every diamond of $G'$ is contained in some gem of $G'$. So, to obtain $G$ from $G'$ it is enough to find all gems contained in $G'$. This can be done in time $\mathcal O(n^4)$. To decide whether $G$ is a line graph of a triangle-free chordless graph or a P-graph it is enough to test whether or not $G$ contains a claw (the line graph of a triangle-free chordless graph does not contain a claw, and a P-graph does).
In case $G$ is a P-graph, we find its special clique $K$ by finding all centers of claws and extend them to a maximal clique (in case there is only one center of claw, say $u$,
then we check whether $u$ is contained in a clique of size 3, and if it is we extend that clique to a maximal clique, and otherwise $K=\{ u\}$).
All this can also be done in time $\mathcal O(n^4)$.

If $G$ is the line graph of a triangle-free chordless graph, then $G'$ is also a line graph, so the maximum weighted stable set of $G'$ can be computed in time $\mathcal O(n^3)$ using Edmonds' algorithm \cite{edmonds}.

If $G$ is a P-graph, then $G'\setminus K$ is a line graph $L(R')$, and hence maximum weight stable set of $G'$ is either contained in $L(R')$, or has exactly one vertex of $K$.
%(as shown above, clique $K$ can be found in time $\mathcal O(n^4)$).
So, it is enough to compute a maximum weight stable set of $L(R')$, and a maximum weight stable set of $G'\setminus N[v]$, for each $v\in K$. Since, for each $v\in K$ the graph $G'\setminus N[v]$ is a line graph, we conclude that using Edmonds' algorithm a maximum weight stable set of $G'$ can be computed in time $\mathcal O(n^4+n\cdot n^3)=\mathcal O(n^4)$.
\end{proof}

\begin{lemma}\label{SSalgD}
There is an algorithm with the following specifications:
\begin{description}
\item[ Input:] A weighted graph $G\in\mathcal D$.
\item[ Output:]
A maximum weight stable set of $G$.
\item[ Running time:] $\mathcal O(n^4m)$.
\end{description}
\end{lemma}

\begin{proof}
Check whether $G$ contains a 2-join (this can be done in time ${\cal O} (n^2m)$ by the algorithm in \cite{fast2j}).
If it does not, then by Theorem \ref{decomposeTW} $G\in \mathcal D_{\textsc{basic}}\subseteq \mathcal D_{\textsc{basic}}^{\textsc{ext}}$,
and hence we compute maximum weight stable set in ${\cal O} (n^4)$ time by Lemma \ref{SSbasicAlg}.

Otherwise, we construct the 2-join decomposition tree $T_G$ (of depth $1\leq p\leq n$) using marker paths of length 3
in ${\cal O} (n^4m)$ time by Lemma \ref{DT-construction}.
By Lemma \ref{new3} all graphs $G_B^1, \ldots ,G_B^p,G^p$ that correspond to the leaves of $T_G$ are in $\mathcal D_{\textsc{basic}}$.
We now reprocess $T_G$.

Let $P^1$ be the marker path used in the construction of $G^1$. We replace $G^1$ by the corresponding gem block $G^{1g}$.
To do this we need to compute the weights $a^1,b^1,c^1,d^1$ that need to be assigned to the vertices of the gem, and this amounts to
computing four weighted stable set problems on $G^1_B$. Since $G^1_B\in \mathcal D_{\textsc{basic}}\subseteq \mathcal D_{\textsc{basic}}^{\textsc{ext}}$
this can be done in ${\cal O} (n^4)$ time by Lemma \ref{SSbasicAlg}.
By Lemma \ref{gem} $\alpha_w (G^{1g})=\alpha_w (G)+d^1$. In all the other graphs that correspond to the vertices of $T_G$ and contain $P^1$,
we extend $P^1$ using weights $a^1,b^1,c^1,d^1$.
We continue this process for $i=2, \ldots ,p$.
So if $P^i$ is the marker path used in construction of $G^i$, we compute the weights needed to transform it into a gem block, by computing four
weighted stable set problems on $G^i_B$ whose paths in ${\cal M}^i_B$ have all already been extended. Since this graph is in
$\mathcal D_{\textsc{basic}}^{\textsc{ext}}$, this can be done in ${\cal O} (n^4)$ time by Lemma \ref{SSbasicAlg}.
In all the graphs that correspond to vertices of $T_G$ that contain $P^i$ we extend $P^i$ using calculated weights $a^i,b^i,c^i,d^i$.

The last graph we reprocess is $G^p$, and let us denote by $G^{pg}$ the graph that we obtain at the end of the reprocessing procedure.
By repeated application of Lemma \ref{gem},
$\alpha_w (G^{pg})=\alpha_w (G)+d^1+\ldots +d^p$, and so we deduce $\alpha_w (G)$.
The proof of Lemma \ref{gem} actually shows how to keep track of a stable set of $G$ whose weight is $\alpha_w (G)$.
Since $p\leq n$, this algorithm can be implemented to run in time
 $\mathcal O(n^4m+n\cdot n^4)=\mathcal O(n^4m)$.
\end{proof}

\begin{theorem}\label{SSalg}
There is an algorithm with the following specifications:
\begin{description}
\item[ Input:] A weighted graph $G\in\mathcal C$.
\item[ Output:]
A maximum weight stable set of $G$.
\item[ Running time:] $\mathcal O(n^6m)$.
\end{description}
\end{theorem}

\begin{proof}
By  Theorem \ref{th:tarjan} we construct the  clique cutset decomposition tree $T$ of $G$ in ${\cal O} (nm)$ time.
So all the leaves of $T$ are graphs from $\mathcal D$.
By using Tarjan's method from \cite{tarjan} to compute a maximum weight stable set of $G$ it is enough to compute $\mathcal O(n^2)$  maximum weight stable sets
on the leaves of $T$ (each one of which can be computed in ${\cal O} (n^4m)$ time by Lemma \ref{SSalgD}).
Therefore, this algorithm can be implemented to run in time
$\mathcal O(nm+n^2\cdot n^4m)=\mathcal O(n^6m)$.
\end{proof}

\section{Vertex coloring}
\label{sec:vCol}

A {\it k-coloring} of a graph $G$ is a function $c:V(G)\rightarrow \{1,2,\ldots,k\}$, such that for every $uv\in E(G)$,  $c(u)\neq c(v)$. We say that $G$ is {\it $k$-colorable} if it admits a $k$-coloring. The {\it chromatic number} of $G$, denoted by $\chi(G)$, is the smallest number $k$ such that $G$ is $k$-colorable.
In this section we give a polynomial-time coloring algorithm for ${\cal C}$ and prove that every $G\in\mathcal C$ is $\max\{3,\omega(G)\}$-colorable.

A graph is {\it sparse} if every edge is incident with at least one vertex of degree at most 2. Note that every sparse graph is chordless. A proper {\it 2-cutset} of a connected graph $G$ is a pair of non-adjacent vertices $a,b$ such that there is a partition $(X,Y,\{a,b\})$ of $V(G)$ with $X$ and $Y$ anticomplete, both $G[X\cup\{a,b\}]$ and $G[Y\cup\{a,b\}]$ contain a path from $a$ to $b$ and neither $G[X\cup\{a,b\}]$ nor $G[Y\cup\{a,b\}]$ is a chordless path. We say that $(X,Y,\{ a,b\} )$ is a {\it split} of this proper 2-cutset. The {\it blocks of decomposition} $G_X$ and $G_Y$ w.r.t.\ this cutset are defined as follows. Block $G_X$ (resp.\ $G_Y$) is the graph obtained by taking $G[X\cup\{a,b\}]$ (resp.\ $G[Y\cup\{a,b\}]$) and adding a new vertex $u$ (resp.\ $v$) complete to $\{a,b\}$ (and anticomplete to the rest).

A decomposition theorem for the class of chordless graphs is proved in \cite{lmt:isk4}. An improvement of this theorem, that is an {\it extreme decomposition} for this class, is proved in \cite{mft:chordless}. We give both results in the following theorem.

\begin{theorem}[\cite{lmt:isk4,mft:chordless}]\label{DecomposChordless}
 If $G$ is a 2-connected chordless graph, then either $G$ is sparse or $G$ admits a proper 2-cutset. Additionally, if $(X,Y,\{ a,b\} )$ is a split of a proper 2-cutset of $G$ such that $|X|$ is minimum among all such splits, then $a$ and $b$ both have at least two neighbors in $X$, and $G_X$ is sparse.
\end{theorem}

A {\it k-edge-coloring} of a graph $G$ is a function $c:E(G)\rightarrow \{1,2,\ldots,k\}$, such that for every two distinct edges with a common vertex, say $uv$ and $uw$,
 $c(uv)\neq c(uw)$. $G$ is {\it $k$-edge-colorable} if it admits a $k$-edge-coloring. The {\it edge-chromatic number} of $G$ is the smallest number $k$ such that $G$ is $k$-edge-colorable.

The edge-coloring of chordless graphs is studied in \cite{mft:chordless}, where the authors obtained the following result.
For a graph $G$, let $\Delta(G)=\max\{\deg(v)\,:\,v\in V(G)\}$ and $\delta(G)=\min\{\deg(v)\,:\,v\in V(G)\}$.

\begin{theorem}[\cite{mft:chordless}]\label{ColorChordless}
Every  chordless graph $G$ is $\{3,\Delta(G)\}$-edge-colorable. Moreover, there is an $\mathcal O(n^3m)$-time algorithm that finds such an edge-coloring.
\end{theorem}

In this section we will prove a variant of the previous theorem (see  Lemma \ref{TwoColoredBasic}). The following result is an important step towards obtaining a $\max\{3,\omega(G)\}$-coloring for our basic classes.

\begin{lemma}[\cite{mft:chordless}]\label{ListChord}
Let $G=(V,E)$ be a sparse graph and suppose that a list $L_{uv}$ of colors is associated with each edge $uv\in E$. Let $S$ be a stable set of $G$ that contains all vertices of $G$ of degree at least 3. Suppose that for every vertex $u\in S$, all edges incident to $u$ receive the same list. If for each edge $uv\in E$, $|L_{uv}|\geq \max\{\deg(u),\deg(v)\}$ and for each edge $uv\in E$ with no end in $S$, $|L_{uv}|\geq 3$, then there is an edge-coloring $c$ of $G$ such that, for each edge $uv\in E$, $c(uv)\in L_{uv}$.
Furthermore, there is an $\mathcal O(nm)$-time algorithm that finds such an edge-coloring $c$.
\end{lemma}

Let $v_1,\ldots,v_k$, where $1\leq k\leq 3$, be some vertices of a branch $B$ of $G$, such that they do not induce a path of length 2. Furthermore, let the list of colors $L_i$, $|L_i|\geq 2$, be associated with $v_i$, for $1\leq i\leq k$, such that if $v_i$ and $v_j$ are adjacent, for some $1\leq i<j\leq k$, then $L_i\cap L_j\neq \emptyset$. Note that branch $B$ can be edge-colored with $\left|\bigcup_{i=1}^k L_i\right|$ colors, so that every edge incident with $v_i$ is colored with a color from $L_i$, for $1\leq i\leq k$. Indeed, if no two of the vertices from the set $\{v_1,\ldots,v_k\}$ are adjacent, then we can color $B$ greedily (starting from one endnode of $B$). If w.l.o.g.\ $v_1$ and $v_2$ are adjacent, then we can obtain the desired edge-coloring by first coloring the edge $v_1v_2$ (with a color from $L_1\cap L_2$) and then greedily coloring the rest of the branch (starting from the other edge incident with $v_1$ and the other edge incident with $v_2$). We say that such an edge-coloring of the branch $B$ is {\em according to the lists} $L_1, \ldots,L_k$.

A vertex $v$ of $G$ is {\em free} if it is of degree 2 and both of its neighbors are also of degree 2.
Vertices $u$ and $v$ of $G$ are {\it parallel} if they are of degree 2, and contained in distinct parallel branches of $G$.
A {\it ring} of $G$ is a hole of $G$ that has at most one vertex that is of degree at least 3 in $G$.
A {\it small theta} of $G$ is an induced subgraph of $G$ isomorphic to $K_{2,3}$ ($K_{2,3}$ is complete bipartite graph whose sides have sizes 2 and 3). Note that if $H$ is a small theta of a sparse graph $G$, then only degree 3 vertices of $H$ can have neighbors in $G\setminus H$.

A set $S=\{v_1, \ldots,v_k\}$, $1\leq k\leq 3$, of vertices is {\it good} if the following hold:
\begin{itemize}
\item[(i)] at most one of $v_1,\ldots,v_k$ is of degree 2 and not free;
\item[(ii)] if $k=3$, then $S$ is not contained in a ring of $G$ of length 5;
\item[(iii)] if $k=3$ and some $v_i\in S$ is of degree 2 and not free, then the two vertices from $S\setminus\{v_i\}$ are not adjacent.
\end{itemize}

\begin{lemma}\label{2list2sparse}
Let $G$ be a triangle-free sparse graph with $\delta(G)=2$, and let $v_1,\ldots,v_k\in V(G)$, $1\leq k\leq 3$, be such that $\{v_1,\ldots,v_k\}$ does not induce a path of length 2. To vertices $v_i$, for $1\leq i\leq k$, the lists of colors $L_i\subseteq\{1,\ldots,s\}$, where $s=\max\{3,\Delta(G)\}$, are associated so that $|L_i|\geq\deg(v_i)$. Furthermore, if $v_i$ and $v_j$ are adjacent, for some $1\leq i<j\leq k$, then $L_i\cap L_j\neq \emptyset$. If one of the following holds:
\begin{itemize}
\item[(1)] $k=3$, $\{v_1,\ldots,v_k\}$ is contained in a ring of length 5 and $\left|L_1\cup L_2\cup L_3\right|\geq 3$;
\item[(2)] $k\leq 3$ and the set $\{v_1,\ldots,v_k\}$ is good;
\item[(3)] $k=2$, and if $v_1$ and $v_2$ are both of degree 2, then $\{v_1,v_2\}$ is not contained in a small theta of $G$;
\end{itemize}
then there is an $s$-edge-coloring of $G$, such that every edge incident with $v_i$ is colored with a color from $L_i$, for $1\leq i\leq k$. Furthermore, there is an $\mathcal O(nm)$-time algorithm that finds such an edge-coloring.
\end{lemma}

\begin{proof}
We prove the result for each of the cases (1), (2) and (3) separately.

\medskip
\noindent
(1) Let $B=u_1u_2u_3u_4u_5u_1$ be the ring of $G$ that contains $\{v_1,\ldots,v_k\}$, and w.l.o.g.\ $v_1=u_1$, $v_2=u_2$ and $v_3=u_4$. Furthermore, let $v$ be
vertex of $B$ with maximum degree. As a first step, we $s$-edge-color $B$.

First, assume that $|L_1\cup L_2|\geq 3$. Then, we color $u_1u_2$ with a color  $c\in L_1\cap L_2$, then color edges $u_1u_5$ and $u_2u_3$ with distinct colors from $L_1\setminus\{c\}$ and $L_2\setminus\{c\}$, respectively, and finally color edges $u_3u_4$ and $u_1u_5$ with distinct colors from $L_3$. So, w.l.o.g.\ let $L_1=L_2=\{1,2\}$. Then we may assume $3\in L_3$, and let $c\in L_3\setminus\{3\}$. Now, we color the edges $u_1u_5$ and $u_2u_3$ with a color  $c'\in\{1,2\}\setminus\{c\}$,  $u_1u_2$ with the color from $\{1,2\}\setminus\{c'\}$, $u_3u_4$ in 3 and $u_4u_5$ in $c$.

So, we have obtained an $s$-edge-coloring of $B$. Let $G'$ be the graph obtained from $G$ by removing all vertices of $B$ except $v$. Now, to complete the $s$-edge-coloring of $G$ we $s$-edge-color $G'$ using Lemma \ref{ListChord} (such that all edges receive the list $\{1,2,\ldots,s\}$), and then permute the colors (in this edge-coloring of $G'$) such that all edges incident with $v$ (in $G$) have different colors.

\medskip
\noindent(2) We prove the claim by induction on $|V(G)|$. If $k=1$, then to obtain an $s$-edge-coloring of $G$ we first $s$-edge-color $G$ by Lemma \ref{ListChord} (such that all edges receive the list $\{1,2,\ldots,s\}$), and then permute the colors such that edges incident with $v_1$ have colors from the list $L_1$. So, we may assume that $k\geq 2$. Also, by induction, we may assume that $G$ is connected. We now consider the following cases.

\medskip
\noindent{\bf Case 1.} $\Delta(G)=2$.
\medskip

\noindent If there is an edge $e$ of $G$ that is not incident with $v_1$, $v_2$ nor $v_3$, then to obtain an edge coloring of $G$ we first edge-color the path $G\setminus\{e\}$ according to the lists $L_1$, $L_2$ and $L_3$ and then color the edge $e$. So, we may assume that every edge of $G$ is incident with at least one of $v_1$, $v_2$ or $v_3$.

Let $k=2$. Then $G$ is of length 4 and vertices $v_1$ and $v_2$ are not adjacent. So, to obtain an edge coloring of $G$ we first color edges incident with $v_1$ (with colors from $L_1$) and then edges incident with $v_2$ (with colors from $L_2$).

Let $k=3$. Since $G$ is not a hole of length 5, we have that $G=u_1u_2u_3u_4u_5u_6u_1$, and no two  vertices from $\{v_1,v_2,v_3\}$ are adjacent. So, we may assume w.l.o.g.\ that $v_1=u_1$, $v_2=u_3$ and $v_3=u_5$. If there is a color $c\in L_1\cap L_2\cap L_3$, then to obtain an edge-coloring of $G$ we first color edges $u_1u_2$, $u_3u_4$ and $u_5u_6$  with $c$, and then color the remaining edges according to the lists $L_1$, $L_2$ and $L_3$. So, let us assume that $L_1\cap L_2\cap L_3=\emptyset$. Then to obtain an edge-coloring of $G$ we first greedily edge-color the path $u_6u_1u_2u_3$ according to the lists $L_1$ and $L_2$. Note that either the colors of $u_6u_1$ and $u_2u_3$ are distinct, or $u_6u_1$ and $u_2u_3$ are colored with the same color which is not in $L_3$. In both cases we can color the remaining edges  of $G$ according to the list $L_3$.
\medskip

\noindent{\bf Case 2.} $v_1$ is contained in a ring.
\medskip

\noindent
By Case 1 we may assume that $\Delta (G) \geq 3$.
Let $B$ be the ring of $G$ that contains $v_1$, let $v$ be the vertex of $B$ of degree at least 3 and let $G'$ be the graph obtained from $G$ by deleting vertices of $B\setminus v$ (and edges incident with these vertices). Note that $G'$ is triangle-free sparse and that $v$ is of degree at least 3, of degree 1 or is free in $G'$. Furthermore, if $v$ is of degree 1 in $G'$, then let $P=v\ldots v'$ be the limb of $G'$ that contains $v$; otherwise $P=\{v\}$ and $v'=v$. Also, let $V''=(V(G)\setminus(V(B)\cup V(P)))\cup\{v'\}$ and $G''=G[V'']$.

In this case we will assume that $k=3$, that is, if $k=2$, then we take $v_3$ to be an arbitrary vertex such that $\{v_1,v_2,v_3\}$ satisfies conditions of this lemma, and $L_3=\{1,2,\ldots,s\}$ (such $v_3$ exists: if $v_2\not\in B$, then we may define $v_3$ to be $v$ or a free vertex of $B$; if $v_2\in B$, then we may define $v_3$ to be a vertex from $V''\setminus\{v\}$ of degree at least 3 or free). It suffices to consider the following cases.

\smallskip
\noindent{\bf Case 2.1.} $\{v_1,v_2,v_3\}\subseteq V(B)$.
\smallskip

\noindent To obtain an $s$-edge-coloring of $G$ we first edge-color $B$ according to the lists $L_1,\ldots,L_k$ (as in Case 1). Then, we $s$-edge-color $G'$ using Lemma \ref{ListChord} (such that all edges receive the list $\{1,2,\ldots,s\}$), and then permute the colors (in this edge-coloring of $G'$) such that all edges incident with $v$ (in $G$) have different colors.
\smallskip

\noindent{\bf Case 2.2.} $v_2\in B$ and $v_3\not\in B$.
\smallskip

\noindent
W.l.o.g. $v_1\neq v$.
First assume that $P=\{v\}$. If  $v_3$ is not adjacent to $v$, then  to obtain a desired edge-coloring of $G$ we first edge-color $B$ according to the lists $L_1$ and $L_2$ (as in Case 1). Let $L$ be the set of colors used for coloring edges incident with $v$ in this coloring. Then, to complete $s$-edge-coloring of $G$ we, by induction, edge-color $G'$ such that the lists $L'$ and $L_{3}$ are associated with vertices $v$  and $v_3$, where $L'=\{1,2,\ldots,s\}\setminus L$ if $v_2\neq v$, or $L'=L_2\setminus L$ if $v=v_2$. So, let us assume that $v_3$ is adjacent to $v$. Then $v_1$ and $v_2$ are not adjacent and not adjacent to $v$, and they are free or $v_2=v$. If  $v=v_2$, then  to obtain a desired edge-coloring of $G$ we first, by induction, edge-color $G'$ such that the lists $L_2$ and $L_{3}$ are associated with vertices $v_2$  and $v_3$, and then edge-color $B$ (as in Case 1) such that the lists $L_1$ and  $\{1,2,\ldots,s\}\setminus L_2'$ are associated with $v_1$ and $v_2$, where $L_2'$ is the set of colors used for coloring edges incident with $v$ in edge-coloring of $G'$. Hence, suppose that $v\neq v_2$. Let $c\in L_3$.  Then  to obtain a desired edge-coloring of $G$ we first edge-color $B$ such that  the lists $L_1$, $L_2$ and $\{1,2,\ldots,s\}\setminus \{c\}$, are associated with $v_1$, $v_2$ and $v$. Then, to complete $s$-edge-coloring of $G$ we, by induction, edge-color $G'$ such that the lists $\{1,2,\ldots,s\}\setminus L$ and $L_{3}$ are associated with vertices $v$  and $v_3$, where $L$ is the set of colors used for coloring edges incident with $v$ in this edge-coloring of $B$.

Next, assume that $P\neq \{v\}$. If $v_3\not\in V(P)$, then we first edge-color $B$ such that the lists $L_1$ and $L_2$ are associated with $v_1$ and $v_2$ (as in Case 1). Then we edge-color $P$ such that the edge incident with  $v$ is colored with a color not used for coloring edges incident with $v$ in this edge-coloring of $B$. Finally, we edge color $G''$ by induction, such that the lists $L''=\{1,2,\ldots,s\}\setminus\{c\}$ and $L_3$ are associated with $v'$ and $v_3$, where $c$ is the color used for coloring edge of $P$ incident with $v'$ (note that $L''\cap L_3\neq\emptyset$, since $|L''|=s-1$). So, suppose that $v_3\in V(P)$. If $v_3$ is not adjacent to $v$, then we first edge-color $B$ such that the lists $L_1$ and $L_2$ are associated with $v_1$ and $v_2$ (as in Case 1). Then we greedily edge-color $P$ from $v$ to $v'$, such that the edge incident with  $v$ is colored with a color not used for coloring edges incident with $v$ in this edge-coloring of $B$, and that edges incident with $v_3$ are colored with colors from $L_3$. To complete edge-coloring of $G$, we edge color $G''$ using Lemma \ref{ListChord} (such that all edges receive the list $\{1,2,\ldots,s\}$), and then permute the colors (in this edge-coloring of $G''$), such that all edges incident with $v'$ (in $G$) have different colors. Finally, suppose that $v$ and $v_3$ are adjacent.
Then $v_3$ is of degree 2 and not free, and so $v_1$ and $v_2$ are not adjacent, $v_1$ is free and $v_2$ is either free or $v_2=v$. In particular, no vertex of
$\{ v_1,v_2\}$ is adjacent to $v$. If $L_2\cap L_3\neq \emptyset$ then
let $c\in L_2\cap L_3$, and otherwise let $c$ be any color from $L_3$.
Then we first edge-color $B$ (as in Case 1), such that: if $v\neq v_2$, then lists $L_1$, $L_2$ and $\{1,2,\ldots,s\}\setminus\{c\}$  are associated $v_1$, $v_2$ and $v$; if $v=v_2$, then lists $L_1$ and $L_2\setminus\{c\}$  are associated $v_1$ and $v_2$. Then we greedily edge-color $P$ such that $vv_3$ is colored with $c$ and the other edge from $P$ incident with $v_3$ with a color from $L_3\setminus\{c\}$. Finally, we edge-color $G''$ using Lemma \ref{ListChord} (such that all edges receive the list $\{1,2,\ldots,s\}$), and then permute the colors (in this edge-coloring of $G''$) such that all edges incident with $v'$ (in $G$) have different colors.
\smallskip

\noindent{\bf Case 2.3.} $v_2,v_3\not\in V(B)$.
\smallskip

\noindent First, assume that $P=\{v\}$ (i.e.\ $v=v'$). If $v_1=v$, then to obtain an $s$-edge-coloring of $G$, we first $s$-edge-color $G'$ by induction, such that the lists $L_1$, $L_2$ and $L_3$ are associated with $v_1$, $v_2$ and $v_3$ (note that $|L_1|\geq 4$, and hence $|L_1\cup L_2\cup L_3|\geq 4$). Finally, we edge-color $B$ and then permute the colors in this edge-coloring of $B$ such that all edges incident with $v$ receive different color. So, suppose $v\neq v_1$. If $v_1$ is not adjacent to $v$, then we first $s$-edge-color $G'$ by induction, such that the lists  $L_2$ and $L_3$ are associated with  $v_2$ and $v_3$, and then edge-color $B$ (as in Case 1) such that the lists $L_1$ and $\{1,2,\ldots,s\}\setminus L$ are associated with $v_1$ and $v$, where $L$ is the set of colors used for coloring edges incident with $v$ in this edge-coloring of $G'$. Finally, suppose that $v_1$ is adjacent to $v$. Then $v_1$ is of degree 2 and not free, so $v_2$ and $v_3$ are either of degree at least 3 or free, and hence $\{v_2,v_3\}$ is anticomplete to $v$.
Let $c$ be a color from $L_1$, and $L'=\{1,2,\ldots,s\}\setminus\{c\}$. Hence, to obtain an $s$-edge-coloring of $G$, we first edge-color  $G'$ by induction, such that the lists $L'$, $L_2$ and $L_3$ are associated with $v$, $v_2$ and $v_3$ (note that $|L'|=s-1\geq 3$, and hence $|L'\cup L_2\cup L_3|\geq 3$), and then edge-color $B$ (as in Case 1) such that the lists $L_1$ and $\{1,2,\ldots,s\}\setminus L''$ are associated with $v_1$ and $v$, where $L''$ is the set of colors used for coloring edges incident with $v$ in this edge-coloring of $G'$.

Now, suppose that $v\neq v'$. If $v_2,v_3\in V''$, then we first, by induction, $s$-edge-color $G''$ such that the lists $L_2$ and $L_3$ are associated with $v_2$ and $v_3$. Then we edge-color $P$ greedily from $v'$ to $v$ (note that $v$ and $v'$ are not adjacent), such that the edge incident with $v'$ is colored with a color not used for coloring edges incident with $v'$ in this edge-coloring of $G''$, and that the edge incident with $v$ is colored with a color from $L_1$. Let this color be $c$. To complete edge-coloring of $G$, we edge-color $B$ such that: if $v_1=v$, then the list $L_1\setminus \{c\}$ is associated with $v_1$; if $v_1\neq v$, then the lists $L_1$ and $L=\{1,2,\ldots,s\}\setminus \{c\}$ are associated with $v_1$ and $v$ (note that $L\cap L_1\neq\emptyset$, since $|L|=s-1$). Next, suppose that $v_2,v_3\in V(P)$. In this case, we first edge-color $P$ such that the lists $L_2$, $L_3$ and possibly $L_1$ (if $v_1=v$) are associated with $v_2$, $v_3$ and possibly $v_1$ (if $v_1=v$). Let $c$ be the color of the edge incident with $v$ in this edge-coloring of $P$. Then we edge-color $B$ such that: if $v_1=v$, then the list $L_1\setminus \{c\}$ is associated with $v_1$; if $v_1\neq v$, then the lists $L_1$ and $L=\{1,2,\ldots,s\}\setminus \{c\}$ are associated with $v_1$ and $v$ (note that $L\cap L_1\neq\emptyset$, since $|L|=s-1$). To complete edge-coloring of $G$, we $s$-edge-color $G''$ using Lemma \ref{ListChord} (such that all edges receive the list $\{1,2,\ldots,s\}$), and then permute the colors (in this edge-coloring of $G''$) such that all edges incident with $v'$ (in $G$) have different colors.
Finally, we may assume that $v_2\in V(P)\setminus\{v'\}$ and $v_3\in V''\setminus\{v'\}$. To obtain $s$-edge-coloring of $G$ we first $s$-edge-color $P$ such that the lists $L_2$ and possibly $L_1$ (if $v_1=v$) are associated with $v_2$ and possibly $v_1$ (if $v_1=v$). Let $c$ (resp.\ $c'$) be the color of the edge incident with $v$ (resp.\ $v'$) in this edge-coloring of $P$. Next, by induction, we $s$-edge-color $G''$ such that the lists $L_3$ and $L'=\{1,2,\ldots,s\}\setminus\{c'\}$ are associated with $v_3$ and $v'$ (note that $L_3\cap L'\neq \emptyset$, since $|L'|=s-1$). To complete $s$-edge-coloring of $G$ we $s$-edge-color $B$ such that: if $v_1=v$, then the list $L_1\setminus \{c\}$ is associated with $v_1$; if $v_1\neq v$, then the lists $L_1$ and $L=\{1,2,\ldots,s\}\setminus \{c\}$ are associated with $v_1$ and $v$ (note that $L\cap L_1\neq\emptyset$, since $|L|=s-1$).
\medskip

By Case 2, from now on we may assume that no vertex from $\{v_1,\ldots,v_k\}$ is contained in a ring, and by Case 1 we may assume that $\Delta (G)\geq 3$.
In particular, since every vertex of $G$ is contained in a ring or a branch, every vertex of $\{ v_1,\ldots ,v_k\}$ is contained in a branch of $G$.

\medskip
\noindent{\bf Case 3.} $v_1$ is free.
\medskip

\noindent Let $B$ be the branch of $G$ that contains $v_1$ with endnodes $u_1$ and $u_2$, and let $G'$ be the graph obtained from $G$ by deleting internal vertices of $B$ (and edges incident with these vertices). Note that since $G$ is sparse, vertices $u_1$ and $u_2$ are free or of degree at least 3 in $G'$, and every neighbor of $u_1$ and $u_2$ is of degree 2 in $G$ and $G'$. In particular, for every $v\in V(G')\setminus\{u_1,u_2\}$, if $\{u_1,u_2,v\}$ is not contained in a ring of length 5 of $G'$, then the set $\{u_1,u_2,v\}$ is good in $G'$.
\smallskip

\noindent{\bf Case 3.1.} Neither $v_2$ nor $v_3$ is adjacent to both $u_1$ and $u_2$.
\smallskip

\noindent First, let us assume that $V(B)\cap \{v_2,\ldots,v_k\}\neq \emptyset$. If $\{v_1,\ldots,v_k\}\subseteq V(B)$, then to obtain an $s$-edge-coloring of $G$ we first $s$-edge-color $B$ according to the lists $L_1,\ldots,L_k$, and then, by induction, $s$-edge-color $G'$ with the list $L_{i}'$ associated with $u_i$, for $i\in\{1,2\}$. The list $L_i'$, for $i\in\{1,2\}$, is defined as follows: $L_{i}'=\{1,2,\ldots,s\}\setminus \{c_{u_i}\}$ if $u_i\not\in\{v_2,\ldots,v_k\}$ and $L_{i}'=L_{j}\setminus \{c_{u_i}\}$ if $u_i=v_j$, for some $2\leq j\leq k$, where $c_{u_1}$ (resp.\ $c_{u_2}$) is the color of the edge incident with $u_1$ (resp.\ $u_2$) in this edge-coloring of $B$.

Let us now assume that w.l.o.g.\ $v_2\in B$, but $v_3\not\in B$ (in this case $k=3$), and w.l.o.g.\ let $v_2$ be in the $u_1v_1$-subpath of $B$. If  $v_2$ and $v_3$ are not adjacent (i.e.\ $v_2\neq u_1$, or $v_2=u_1$ and $v_3$ is not adjacent to $u_1$), then  to obtain desired $s$-edge-coloring of $G$ we first $s$-edge-color $B$ according to the lists $L_1$ and $L_2$, and such that the edges incident with $u_1$ and $u_2$ receive different colors (this can be done since the edge incident with $u_2$ is the last that we color, and we have at least 2 options for coloring it). Then, by induction, we $s$-edge-color $G'$ such that the lists $L_{1}'$, $L_{2}'$ and $L_3$ are associated with vertices $u_1$, $u_2$ and $v_3$ ($L_{1}'$ and $L_{2}'$ are defined as in the previous part of the proof, and they satisfy $|L_1'\cup L_2'|\geq 3$, since $L_1'\neq L_2'$). So, let us assume that $v_3$ is adjacent to $u_1$ and $v_2=u_1$. Let $c\in L_2\cap L_3$. To obtain desired $s$-edge-coloring of $G$, we first greedily $s$-edge-coloring $B$ starting with the edge incident with $u_1$ and giving it a color $c'\in L_2\setminus\{c\}$, and such that the color of the edge incident with $u_2$ is not $c'$. Then, by induction, we $s$-edge-color $G'$ such that the lists $L_{1}''=L_2\setminus\{c'\}$, $L_{2}''=\{1,2,\ldots,s\}\setminus\{c_{u_2}\}$ and $L_3$ are associated with vertices $u_1$, $u_2$ and $v_3$, where $c_{u_2}$ is the color of the edge incident with $u_2$ in the edge-coloring of $B$ (note that $|L_1''\cup L_2''|\geq 3$, since $L_1''\neq L_2''$).

Finally, let us assume that $v_2$ and $v_3$ are not in $B$.
Observe that $\{ u_1,v_1,v_2,u_2\}$ cannot induce a path, since otherwise both $v_2$ and $v_3$ would be of degree 2 and not free. Therefore, by the case we are in,
at most one of the sets $\{v_2,v_3,u_1\}$ and $\{v_2,v_3,u_2\}$ induces a path of length 2. W.l.o.g.\ assume that $\{v_2,v_3,u_1\}$ does not induce a path of length 2. Furthermore, if $u_1$ is adjacent to $v_2$ or $v_3$, then that vertex is not free or degree at least 3 in $G$.
Also, by the case we are in, $\{ v_2,v_3,u_1\}$ cannot be contained in a ring of $G'$ of length 5.
 Hence the set $\{v_2,v_3,u_1\}$ is good in $G'$. Now, to obtain desired $s$-edge-coloring of $G$ we first, by induction, $s$-edge-color $G'$ such that the lists $\widetilde{L}_{1}$, $L_{2}$ and possibly $L_3$ (if $k=3$) are associated with vertices $u_1$, $v_2$ and possibly $v_3$ (if $k=3$), where $\widetilde{L}_1=\{1,2,\ldots,s\}\setminus\{c_1'\}$ ($c_1'\not\in L_1$ if $|L_1|=2$, or arbitrary otherwise). Then branch $B$ is greedily 3-edge-colored in the following way: we color the edge incident with $u_1$ with color $c_1'$, the edge incident with $u_2$ with a color not used for coloring edges incident with $u_2$ in $G'$, then color $v_1u_2$-subpath of $B$ (greedily from $u_2$ to $v_1$) and finally color $u_1v_1$-subpath of $B$ (greedily from $v_1$ to $u_1$).

\smallskip

\noindent{\bf Case 3.2.} $v_2$ is adjacent to both $u_1$ and $u_2$.
\smallskip

\noindent In this case, $v_2$ is of degree 2 and not free in $G$. Also, if $k=3$, then $v_3$ must be free or of degree at least 3 in $G$, and it follows that $v_3$ is anticomplete
to $\{ u_1,u_2\}$ and so is free or of degree at least 3 in $G'$.

First, assume that $k=2$ or $v_3\not\in B$. Note that in this case, if $k=3$, then $\{u_1,v_2,v_3\}$ is not contained in a ring of $G'$ of length 5. So, to obtain desired $s$-edge-coloring of $G$ we do the following. First, by induction, we $s$-edge-color $G'$ so that the lists $\widetilde{L}_{1}$, $L_{2}$ and possibly $L_3$ (if $k=3$) are associated with vertices $u_1$, $v_2$ and possibly $v_3$ (if $k=3$), where $\widetilde{L}_1=\{1,2,\ldots,s\}\setminus\{c_1'\}$ ($c_1'\not\in L_1$ if $|L_1|=2$, or arbitrary otherwise). Then branch $B$ is greedily 3-edge-colored in the following way: we color the edge incident with $u_1$ with color $c_1'$, the edge incident with $u_2$ with a color not used for coloring edges incident with $u_2$ in $G'$, then color $v_1u_2$-subpath of $B$ (greedily from $u_2$ to $v_1$) and finally color $u_1v_1$-subpath of $B$ (greedily from $v_1$ to $u_1$).

Next let us assume that $v_3$ is in $B$ and free. Then $v_1$ and $v_3$ are not adjacent, and let us w.l.o.g.\ assume that $v_3$ is in the $v_1u_2$-subpath of $B$. Then to obtain desired $s$-edge-coloring of $G$ we first, by induction, $s$-edge-color $G'$ such that the lists $\widetilde{L}_1$ and $L_2$ are associated with vertices $u_1$ and $v_2$, where $\widetilde{L}_1=\{1,2,\ldots,s\}\setminus\{c_1'\}$ ($c_1'\not\in L_1$ if $|L_1|=2$, or arbitrary otherwise). Then branch $B$ is greedily 3-edge-colored in the following way: we color the edge incident with $u_1$ with color $c_1'$, the edge incident with $u_2$ with a color not used for coloring edges incident with $u_2$ in $G'$, then color $v_1u_2$-subpath of $B$ (greedily from $u_2$ to $v_1$) and finally color $u_1v_1$-subpath of $B$ (greedily from $v_1$ to $u_1$).

So, w.l.o.g.\ let $v_3=u_1$. Then to obtain desired $s$-edge-coloring of $G$ we first, by induction, $s$-edge-color $G''=G\setminus\{v_2\}$, such that the lists  $L_1$, $L_3''$ and $L''$ are associated with vertices $v_1$, $v_3$ and $u_2$, where $L_3''=L_3\setminus\{c'\}$ for some $c'\in L_2\cap L_3$, and $L''=\{1,2,\ldots,s\}\setminus\{c''\}$ for some $c''\in L_2\setminus\{c'\}$ (note that $\{v_1,v_3,u_2\}$ is not contained in ring of $G''$ of length 5). Finally, we color the edge $u_1v_2$ in $c'$ and $u_2v_2$ in $c''$.
\medskip

By Case 3, from now on we may assume that no vertex from $\{v_1,\ldots,v_k\}$ is free. Therefore it suffices to consider the following cases.

\medskip
\noindent{\bf Case 4.} $v_1$ and possibly $v_3$ (if $k=3$) are of degree at least 3.
\medskip

\noindent If $v_2$ is also of degree at least 3, then the proof follows from Lemma \ref{ListChord} (with $S$ the set of all vertices of degree at least 3, lists $L_i$, for $i\in\{1,\ldots,k\}$, given to edges incident with $v_i$, and list $\{1,\ldots,s\}$ given to all other edges). So, suppose that $\deg(v_2)=2$. Let $B=u_1\ldots u_2$ be the branch of $G$ that contains $v_2$, and $G'$ be the graph obtained from $G$ by deleting internal vertices of $B$ (and edges incident with these vertices). Note that since $G$ is sparse $\{v_1,v_3\}$ is anticomplete to $\{u_1,u_2\}$ and each of the vertices $u_1$ and $u_2$ is free or of degree at least 3 in $G'$.

First, let us assume that $\Delta(G)=4$. If $v_1,v_3\not\in B$, then to obtain $s$-edge-coloring of $G$, we first $s$-edge-color $B$ according to $L_2$. Let $L'=\{1,2,\ldots,\Delta(G)\}\setminus\{c_{u_1}\}$ and $L''=\{1,2,\ldots,\Delta(G)\}\setminus\{c_{u_2}\}$, where $c_{u_1}$ (resp.\ $c_{u_2}$) is the color of the edge incident with $u_1$ (resp.\ $u_2$) in this edge-coloring of $B$. Finally, we $s$-edge-color $G'$ using Lemma \ref{ListChord}, such that the lists  $L_1$, $L'$, $L''$ and possibly $L_3$ (if $k=3$) are associated with edges incident with $v_1$, $u_1$, $u_2$ and possibly $v_3$ (if $k=3$), respectively, and the list $\{1,2,\ldots,\Delta(G)\}$ associated with all other edges. If w.l.o.g.\ $v_1=u_1$, then
to obtain $s$-edge-coloring of $G$, we first $s$-edge-color $B$ according to the lists $L_1$, $L_2$ and possibly $L_3$ (if $v_3=u_2$). Next, we associate with $v_1$ the list $L_1'=L_1\setminus\{c_{u_1}\}$, and to $u_2$ the list $L''=\{1,2,\ldots,\Delta(G)\}\setminus\{c_{u_2}\}$ if $v_3\neq u_2$, or $L''=L_3\setminus\{c_{u_2}\}$ if $v_3=u_2$, where $c_{u_1}$ (resp.\ $c_{u_2}$) is the color of the edge incident with $u_1$ (resp.\ $u_2$) in this edge-coloring of $B$. Then we $s$-edge-color $G'$, by induction, such that the lists $L_1'$, $L''$ and possibly $L_3$ (if $k=3$ and $v_3\neq u_2$) are associated with $u_1$, $u_2$ and possibly $v_3$ (if $k=3$ and $v_3\neq u_2$).

Finally, let $\Delta(G)=3$. In this case $L_1=L_3=\{1,2,3\}$, and hence any 3-edge-coloring of $G$ respects the lists $L_1$ and $L_3$. So, to obtain a desired $s$-edge-coloring of $G$ we first $s$-edge-color $G$ using Lemma \ref{ListChord} (we give the list $\{1,2,3\}$ to all edges) and then permute the colors such that the edges incident with $v_2$ receive colors from the list $L_2$.
\medskip

\noindent(3) We prove the claim by induction on $|V(G)|$. By induction, we may assume that $G$ is connected.
It suffices to consider the following cases.
\medskip

\noindent{\bf Case 1.} $v_1$ and $v_2$ are of degree at least 3.
\medskip

\noindent The proof in this case follows from Lemma \ref{ListChord} (with $S$ the set of all vertices of degree at least 3, lists $L_i$, for $i\in\{1,2\}$, given to edges incident with $v_i$, and list $\{1,2,\ldots,s\}$ given to all other edges).
\medskip

\noindent{\bf Case 2.} $v_1$ is of degree 2.
\medskip

\noindent
If $\Delta (G)=2$ then $G$ is a hole and it is easy to see how to obtain the desired coloring. So we may assume that $\Delta (G)\geq 3$.
We now consider the following cases.
\medskip

\noindent{\bf Case 2.1.} $v_1$ is contained in a ring of $G$.
\medskip

\noindent Let $B$ be that ring, let $v$ be the  vertex of degree at least 3 of $B$, and let $G'$ be the graph obtained from $G$ by deleting degree 2 vertices of $B$ (and edges incident with these vertices). Note that $G'$ is triangle-free sparse and that $v$ is of degree at least 3, of degree 1 or is free in $G'$. Also, since $G$ is sparse, $v$ is not adjacent to a vertex of degree at least 3. In particular, if $v$ is contained in a small theta of $G$ (or any of its induced subgraphs), then $v$ is not a degree 2 vertex of this small theta. Finally, if $v$ is of degree 1 in $G'$, then let $P=v\ldots v'$ be the limb of $G'$ that contains $v$; otherwise $P=\{v\}$ and $v=v'$. Also, let $V''=(V(G)\setminus(V(B)\cup V(P))\cup\{v'\}$ and $G''=G[V'']$.

If $\{v_1,v_2\}\in V(B)$, then we proceed as in Case 2.1 of part (2). So, let us assume that $v_2\not\in V(B)$. Our proof in this case is similar to the proof of Case 2.3 of part (2).

First, let $P=\{v\}$ (i.e.\ $v=v'$). If $v_1=v$, then to obtain an $s$-edge-coloring of $G$, we first $s$-edge-color $G'$ by induction such that the lists $L_1$ and $L_2$ are associated with $v_1$ and $v_2$. To complete edge-coloring of $G$, we edge-color $B$ and then permute the colors in this edge-coloring of $B$ such that all edges incident with $v$ receive different color. So, suppose $v\neq v_1$. If $v_1$ is not adjacent to $v$, then we first $s$-edge-color $G'$ (using part (2)) such that the list $L_2$ is associated with $v_2$, and then edge-color $B$ (as in Case 1 of (2)) such that the lists $L_1$ and $\{1,2,\ldots,s\}\setminus L$ are associated with $v_1$ and $v$, where $L$ is the set of colors used for coloring edges incident with $v$ in this edge-coloring of $G'$.  So, suppose that $v_1$ is adjacent to $v$. Then we first $s$-edge-coloring $G'$ by induction such that the lists  $L_2$ and $L'=\{1,2,\ldots,s\}\setminus\{c\}$ are associated with  $v_2$ and $v$, where $c\in L_1$ is arbitrary (note that $L_2\cap L'\neq\emptyset$, since $|L'|=s-1$). To complete edge-coloring of $G$ we edge-color $B$ (as in Case 1 of (2)) such that the lists $L_1$ and $\widetilde{L}=\{1,2,\ldots,s\}\setminus L''$ are associated with $v_1$ and $v$, where $L''$ is the list of colors used for coloring edges incident with $v$ in this edge-coloring of $G'$ (note that $c\in L_1\cap \widetilde{L}$).

Suppose now that $v\neq v'$. If $v_2\in V''$, then we first $s$-edge-color $G'$ (using part (2)) such that the list $L_2$ is associated with $v_2$ . Then we edge-color $P$ greedily from $v'$ to $v$ (note that $v$ and $v'$ are not adjacent), such that the edge incident with $v'$ is colored with a color not used for coloring edges incident with $v'$ in this edge-coloring of $G''$, and that the edge incident with $v$ is color with a color from $L_1$. Let this color be $c$. To complete edge-coloring of $G$, we edge color $B$ such that: if $v_1=v$, then the list $L_1\setminus \{c\}$ is associated with $v_1$; if $v_1\neq v$, then the lists $L_1$ and $L=\{1,2,\ldots,s\}\setminus \{c\}$ are associated with $v_1$ and $v$ (note that $L\cap L_1\neq\emptyset$, since $|L|=s-1$). Next, suppose that $v_2\in V(P)$. In this case, we first edge-color $P$ such that the lists $L_2$ and possibly $L_1$ (if $v_1=v$) are associated with $v_2$  and possibly $v_1$ (if $v_1=v$). Let $c$ be the color of the edge incident with $v$ in this edge-coloring of $P$. Then we edge-color $B$ such that: if $v_1=v$, then the list $L_1\setminus \{c\}$ is associated with $v_1$; if $v_1\neq v$, then the lists $L_1$ and $L=\{1,2,\ldots,s\}\setminus \{c\}$ are associated with $v_1$ and $v$ (note that $L\cap L_1\neq\emptyset$, since $|L|=s-1$). To complete edge-coloring of $G$, we $s$-edge-color $G''$ using Lemma \ref{ListChord} (such that all edges receive the list $\{1,2,\ldots,s\}$), and then permute the colors (in this edge-coloring of $G''$) such that all edges incident with $v'$ (in $G$) have different colors.
\medskip

By Case 2.1, from now on we may assume that neither $v_1$ nor $v_2$ is contained in a ring of $G$.  Let $B=u_1\ldots u_2$ be the branch of $G$ that contains $v_1$, and let $G'$ be the graph obtained from $G$ by deleting internal vertices of $B$ (and edges incident with these vertices). Since $G$ is sparse vertices $u_1$ and $u_2$ are free or of degree at least 3 in $G'$, and every neighbor of $u_1$ and $u_2$ is of degree 2 in $G$ and $G'$. In particular, if $u_1$ (resp.\ $u_2$) is contained in a small theta of $G$ (or any of its induced subgraphs), then $u_1$ (resp.\ $u_2$) is not a degree 2 vertex of this small theta.
\smallskip

\noindent{\bf Case 2.2.} $v_1$ and $v_2$ are not parallel.
\smallskip

\noindent In this case $\{u_1,u_2,v_2\}$ is not contained in a ring of $G'$ of length 5.

If $v_2\in B$, then to obtain $s$-edge-coloring of $G$ we first $s$-edge-color $B$ according to the lists $L_1$ and $L_2$, and then by induction $s$-edge-color $G'$ such that the list $L'$ and $L''$ are associated with $u_1$ and $u_2$. Lists are defined as follows: $L'=\{1,2,\ldots,s\}\setminus \{c_{u_1}\}$ (resp.\ $L''=\{1,2,\ldots,s\}\setminus \{c_{u_2}\}$) if $u_1\neq v_2$ (resp.\ $u_2\neq v_2$) or $L'=L_{2}\setminus \{c_{u_1}\}$ (resp.\ $L''=L_{2}\setminus \{c_{u_2}\}$) if $u_1=v_2$ (resp.\ $u_2=v_2$), where $c_{u_1}$ (resp.\ $c_{u_2}$) is the color of the edge incident with $u_1$ (resp.\ $u_2$) in the edge-coloring of $B$.

So, let $v_2\not\in B$. Since $\{u_1,u_2,v_2\}$ does not induce a path of length 2 (by the case we are in),and since the set is good in $G'$, to obtain $s$-edge-coloring of
$G$ we first $s$-edge-color $B$ according to the list $L_1$, and then by part (2) $s$-edge-color $G'$ with the lists $L'=\{1,2,\ldots,s\}\setminus \{c_{u_1}\}$, $L''=\{1,2,\ldots,s\}\setminus \{c_{u_2}\}$ and $L_2$ associated with $u_1$, $u_2$ and $v_2$, where $c_{u_1}$ (resp.\ $c_{u_2}$) is the color of the edge incident with $u_1$ (resp.\ $u_2$) in the edge-coloring of $B$ (note that $L_2\cap L'\neq\emptyset$ and $L_2\cap L''\neq\emptyset$).
\smallskip

\noindent{\bf Case 2.3.} $v_1$ and $v_2$ are parallel.
\smallskip

\noindent If $v_1$ or $v_2$ is free, then we can apply part (2). So, suppose that both $v_1$ and $v_2$ are not free. Let $B'$ be the branch of $G$ that contains $v_2$.
\smallskip

\noindent{\bf Case 2.3.1.} At least one of the branches $B$ and $B'$ is of length at least 3.
\smallskip

\noindent W.l.o.g.\ let $B'$ be of length at least 3 and $v_2$ adjacent to $u_1$. Now we define colors $c_1$ and $c_2$ that are going to be used when edge-coloring $G$:
\begin{itemize}
  \item if $v_1$ is adjacent to both $u_1$ and $u_2$, then $c_1$ and $c_2$ are distinct colors from $L_1$;
  \item if $v_1$ is adjacent to $u_1$, but not adjacent to $u_2$, then $c_1$ is a color from $L_1$ and $c_2$ a color not from $\{c,c_1\}$, where $c$ is a color from $L_1$ distinct from $c_1$;
  \item if $v_1$ is adjacent to $u_2$, but not adjacent to $u_1$, then $c_2$ is a color from $L_1$ and $c_1$ a color not from $\{c,c_2\}$, where $c$ is a color from $L_1$ distinct from $c_2$.
\end{itemize}
Now, we first, by part (2), $s$-edge-color $G'$ such that the lists $L'=\{1,2,\ldots,s\}\setminus\{c_1\}$, $L''=\{1,2,\ldots,s\}\setminus\{c_2\}$ and $L_2$  are associated with $u_1$, $u_2$ and $v_2$ (note that $|L'\cup L''\cup L_2|\geq 3$, since $L'\cup L''=\{1,2,\ldots,s\}$). To complete the edge-coloring of $G$ we color the branch $B$ in the following way: we first color the edges incident with $u_1$ and $u_2$ with colors $c_1$ and $c_2$, respectively, and then greedily edge-color the rest of $B$ starting from $v_1$ and according to the list $L_1$.
\smallskip

\noindent{\bf Case 2.3.2.} Branches $B$ and $B'$ are of length 2.
\smallskip

\noindent First, let us assume that both $u_1$ and $u_2$ are of degree at least 4 in $G$, and let $G''=G\setminus\{v_1,v_2\}$. Then $G''$ is triangle-free sparse and vertices $u_1$ and $u_2$ are free or of degree at least 3 in $G''$. Now we define colors $c_1,c_2,c_3,c_4$ that are going to be used when edge-coloring $G$:
\begin{itemize}

\item if $|L_1\cap L_2|\geq 2$ and $c',c''\in L_1\cap L_2$, then $c_1=c_4=c'$, $c_2=c_3=c''$;

\item if $|L_1\cap L_2|=1$ and $L_1\cap L_2=\{c\}$, then $c_1=c_4=c$, $c_2=c'$ and $c_3=c''$, where $c'\in L_2\setminus\{c\}$ and $c''\in L_1\setminus\{c\}$;

\item if $L_1\cap L_2=\emptyset$, then $c_1,c_3\in L_1$, $c_1\neq c_3$, and $c_2,c_4\in L_2$, $c_2\neq c_4$.
\end{itemize}
Now, by induction, we edge-color $G''$ such that the lists $\{1,2,\ldots,s\}\setminus\{c_1,c_2\}$ and $\{1,2,\ldots,s\}\setminus\{c_3,c_4\}$  are associated with $u_1$ and $u_2$, and then color edges $u_1v_1$, $u_1v_2$, $u_2v_1$ and $u_2v_2$ in colors $c_1$, $c_2$, $c_3$ and $c_4$, respectively.

So, we may assume that w.l.o.g.\ $\deg_G(u_1)=3$. Let $\widetilde{G}$ be the graph obtained from $G\setminus\{v_1,v_2\}$ by adding the edge $u_1u_2$. Since $u_1$ is of degree 2 in $\widetilde{G}$, graph $\widetilde{G}$ is sparse, and since $\{v_1,v_2\}$ is not contained in a small theta of $G$  graph $\widetilde{G}$ is triangle-free. Furthermore, since at least one neighbor of $u_1$ in $\widetilde{G}$ is of degree 2, $u_1$ is not a degree 2 vertex of some small theta of $\widetilde{G}$.

We define lists of colors  $\widetilde{L}_1$ and $\widetilde{L}_2$ that are going to be used for obtaining an edge-coloring of $\widetilde{G}$ (and $G$):
\begin{itemize}
  \item[(i)] if $|L_1\cap L_2|\geq 2$ and $c_1,c_3\in L_1\cap L_2$, then $\widetilde{L}_1=\{c_1,c_2\}$ and $\widetilde{L}_2=\{1,2,\ldots,s\}\setminus\{c_3\}$, where $c_2\not\in\{c_1,c_3\}$;

\item[(ii)] if $|L_1\cap L_2|=1$ and $L_1\cap L_2=\{c_1\}$, then $\widetilde{L}_1=\{c_1,c_2\}$ and $\widetilde{L}_2=\{1,2,\ldots,s\}\setminus\{c_2\}$, where $c_2\in L_1\setminus\{c_1\}$;

\item[(iii)] if $L_1\cap L_2=\emptyset$, then $\widetilde{L}_1=\{c_1,c_2\}$ and $\widetilde{L}_2=\{1,2,\ldots,s\}\setminus\{c_2\}$, where $c_1\in L_1$  and $c_2\in L_2$.
\end{itemize}
Now, by induction, we $s$-edge-color $\widetilde{G}$ such that the lists $\widetilde{L}_1$ and $\widetilde{L}_2$  are associated with $u_1$ and $u_2$. Furthermore, we can permute the colors $c_1$ and $c_2$ in case (i), such that the edge $u_1u_2$ is colored with $c_1$. Finally, to obtain an edge-coloring of $G$ we extend the obtained edge-coloring of $\widetilde{G}\setminus\{u_1u_2\}$ in the following way. In case (i) we color the edges $u_1v_1$, $u_1v_2$, $u_2v_1$ and $u_2v_2$ with colors $c_1$, $c_3$, $c_3$ and $c_1$, respectively; in case (ii) we color the edges $u_1v_1$, $u_1v_2$, $u_2v_1$ and $u_2v_2$ with colors $c_1$, $c_3$, $c_2$ and $c_1$, respectively, where $c_3\in L_2\setminus\{c_1\}$; in case (iii) we color the edges $u_1v_1$, $u_1v_2$, $u_2v_1$ and $u_2v_2$ with colors $c_3$, $c_4$, $c_1$ and $c_2$, respectively, where $c_3\in L_1\setminus\{c_1\}$ and $c_4\in L_2\setminus\{c_2\}$.
\medskip

Note that this proof yields an $\mathcal O(nm)$-time algorithm that finds described edge-coloring. Indeed, all steps in the proof can be done in linear time, except when Lemma \ref{ListChord} is applied (which takes $\mathcal O(nm)$), but then the edge-coloring of $G$ can be completed in linear time.
\end{proof}
%
%We say that a vertex $v$ of $G$ is {\it tight} if there exists a vertex $v'$, such that $\deg(v)=\deg(v')=2$ and $N_G(v)=N_G(v')$.
%
%A {\em loose limb} in a graph $G$ is a path of length at least~1
%whose internal vertices are of degree 2 in $G$ and whose endnodes are
%of degree 1.
%
%\begin{lemma}\label{TwoColoredBasic}
%Let $G$ be a line graph of a connected triangle-free chordless graph, and $K_1$ and $K_2$ be distinct inclusion-wise maximal cliques of $G$. Suppose that we are given two lists of colors $L_1,L_2\in\{1,2,\ldots,s\}$, where $s=\max\{3,\omega(G)\}$, such that $|L_1|\geq |K_1|$ and $|L_2|\geq|K_2|$, and if $K_1\cap K_2\neq\emptyset$, then $L_1\cap L_2\neq\emptyset$.  Then there exists an $s$-coloring of $G$ such that every vertex of $K_i$ is colored with a color from $L_i$, for $i\in\{1,2\}$.
%\end{lemma}

\begin{lemma}\label{TwoColoredBasic}
Let $G$ be a triangle-free chordless graph, and let $v_1$ and $v_2$ be distinct vertices of $V(G)$ both of degree at least 1.
Let $\widetilde{G}$ be a graph obtained from $G$ by adding a path
$Q=q_1\ldots q_k$, $k\geq 2$, (whose vertices are  disjoint from vertices of $G$) and edges $q_1v_1$ and $q_kv_2$ (these are the only edges between $G$ and $Q$).
Assume that $\widetilde{G}$
 is also triangle-free chordless. Suppose that we are given two lists of colors $L_1,L_2\subseteq\{1,2,\ldots,s\}$, where $s=\max\{3,\Delta(G)\}$, such that $|L_1|\geq \deg_G(v_1)$,
  $|L_2|\geq\deg_G(v_2)$, and if $v_1$ and $v_2$ are adjacent, then $L_1\cap L_2\neq\emptyset$. Also, suppose that if both $v_1$ and $v_2$ are of degree 1 in $G$
  and $|L_1|=|L_2|=1$, then their neighbors in $G$ are distinct.  Then there exists an $s$-edge-coloring of $G$ such that every edge of $G$ incident with $v_i$ is colored with a color from $L_i$, for $i\in\{1,2\}$.
  Furthermore, such an edge coloring can be obtained in ${\cal O} (n^3m)$-time.
\end{lemma}

\begin{proof}
We prove this lemma by induction on $|V(G)|$. By induction, Theorem \ref{ColorChordless} and Lemma \ref{2list2sparse}, we may assume that $G$ is connected.
\medskip

\noindent{\bf Case 1.} $G$ contains a vertex $v$ of degree 1.
\medskip

\noindent First, suppose that $G$ is a path, i.e.\ $P=v\ldots v'$. If $\{v_1,v_2\}\cap\{v,v'\}=\emptyset$, then we $s$-edge-color this path according to the lists $L_1$ and $L_2$. If $|\{v_1,v_2\}\cap\{v,v'\}|=1$ and w.l.o.g.\ $v=v_1$, then we first color the edge incident with $v$ (with a color from $L_1$ if $v_1v_2\not\in E(G)$, or a color from $L_1\cap L_2$ is $v_1v_2\in E(G)$) and then greedily $s$-edge-color the rest of $G$ (starting from $v$) according to the list $L_2$. If w.l.o.g.\ $v_1=v$ and $v_2=v'$, we first color edges incident with $v_1$ and $v_2$ (with colors from $L_1$ and $L_2$ if $v_1v_2\not\in E(G)$, or a color from $L_1\cap L_2$ if $v_1v_2\in E(G)$), and then greedily edge-color the rest of $G$.

So, suppose that $G$ is not a path. Let $B=v\ldots v'$ be the limb of $G$ that contains $v$ and let $G'$ the graph induced by $(V(G)\setminus V(B))\cup\{v'\}$. If $\{v_1,v_2\}\subseteq V(B)$, then we first $s$-edge-color $B$ in the following way: if $B$ is not of length 2 or  $\{v_1,v_2\}\neq\{v,v'\}$, then we $s$-edge-color $B$ as in the previous paragraph;  if $B$ is of length 2 and w.l.o.g\ $v_1=v$ and $v_2=v'$, then we color the edge incident with $v$ with a color $c\in L_1$ and then color the edge incident with $v'$ with a color from $L_2\setminus\{c\}$ (note that $|L_2|\geq 3$). To complete edge-coloring of $G$ we $s$-edge-color $G'$ using Theorem \ref{ColorChordless} and permute the colors in this edge-coloring of $G'$ so that the edges incident with $v'$ (in $G$) all receive different colors.

If $\{v_1,v_2\}\subseteq V(G')$, then we first, by induction, $s$-edge-color $G'$ so that the lists $L_1$ and $L_2$ are associated with $v_1$ and $v_2$ (note that $v'$ is of degree at least 2 in $G'$, so a vertex is of degree 1 in $G'$ iff it is of degree 1 in $G$). Then, we greedily $s$-edge-color $B$ (starting from $v'$) such that edges incident with $v'$ all receive different colors.

Finally, suppose w.l.o.g.\ that $v_1\in V(B)\setminus\{v'\}$ and $v_2\in V(G')\setminus\{v'\}$. If both $v_1$ and $v_2$ are of degree 1 and adjacent to $v'$, then we first color edges incident with $v_1$ and $v_2$ (with colors from $L_1$ and $L_2$), then $s$-edge-color $G\setminus\{v_1,v_2\}$ using Theorem \ref{ColorChordless} and finally permute the colors in this edge-coloring of $G\setminus\{v_1,v_2\}$ such that edges incident with $v'$ all receive different colors. So, suppose that w.l.o.g.\ $v_2$ is of degree at least 2 or not adjacent to $v'$. Then, to obtain an $s$-edge-coloring of $G$, we first greedily $s$-edge-color $B$ such that edges incident with $v_1$ receive colors from the list $L_1$. Let $L'=\{1,2,\ldots,s\}\setminus\{c\}$, where $c$ is the color of the edge incident with $v'$ in this edge-coloring of $B$. Also, let $Q''$ be the path induced by $V(Q)$ and vertices of the $v_1v'$-subpath of $B$, and let $Q'$ be the path induced by $V(Q'')\setminus\{v'\}$. Then $Q'$ is disjoint from $G'$, its endnodes are adjacent to $v'$ and $v_2$, and the graph induced by $V(G')\cup V(Q')$ is triangle-free chordless. Hence, to complete $s$-edge-coloring of $G$ we, by induction, $s$-edge-color $G'$ so that the  lists $L_2$ and $L'$ are associated with $v_2$ and $v'$ (note that $\deg_{G'}(v')\geq 2$, and if $v'v_2\in E(G)$, then $\deg_{G'}(v_2)\geq 2$ and hence $L'\cap L_2\neq\emptyset$ since $|L'|=s-1$).
\medskip

By Case 1, we may assume that $\delta(G)\geq 2$. By Theorem \ref{DecomposChordless}, it is enough to consider the following cases.

\medskip
\noindent{\bf Case 2.} $G$ is sparse.
\medskip

\noindent Follows from part (3) of Lemma \ref{2list2sparse}. Indeed, in this case $v_1$ and $v_2$ are not degree 2 vertices that belong to a small theta $H$ of $G$, since
otherwise $\widetilde{G}[V(H)\cup V(Q)]$ is a cycle with chords.

\medskip
\noindent{\bf Case 3.} $G$ has a cutvertex.
\medskip

\noindent Let $v$ be a cutvertex of $G$ and let $(X_1,X_2,\{v\})$ be a partition of $V(G)$ such that $X_1$ is anticomplete to $X_2$.
%Since $\delta(G)\geq 2$, we may assume that $u$ has at least two neighbors in w.l.o.g.\ $X_1$.

First, let us assume that $v_i\in X_i$, for $i\in\{1,2\}$. Let $Q_i'$, for $i\in\{1,2\}$, be a chordless path between $v_{3-i}$ and $v$ in $G$, and let $Q_i$ be the path induced
in $\widetilde{G}$ by $(V(Q)\cup V(Q_i'))\setminus\{v\}$. Then $Q_i$ is disjoint from $G[X_{i}\cup\{v\}]$, its endnodes are adjacent to $v$ and $v_{i}$ respectively, and the graph induced by $X_{i}\cup\{v\}\cup V(Q_i)$ is triangle-free chordless. So, if $v_i$ is not adjacent to $v$, for some $i\in\{1,2\}$, then we obtain an $s$-edge-coloring  of $G$ in the following way. We first
$s$-edge-color $G[X_{3-i}\cup\{v\}]$ using Theorem \ref{ColorChordless} and then permute the colors in this edge-coloring so that edges incident with $v_{3-i}$ receive colors from the list $L_{3-i}$. Let $L'$ be the set of colors used for coloring edges incident with $v$ in this coloring, and let $L=\{1,2,\ldots,s\}\setminus L'$. Then, by induction, we $s$-edge-color $G[X_i\cup\{v\}]$ so that edges incident with $v_i$ receive colors from the list $L_i$ and edges (from $G[X_i\cup\{v_i\}]$) incident with $v$ receive colors from the list $L$. Merging these colorings we obtain an $s$-coloring of $G$. So, let us assume that $v_i$ is adjacent to $v$, for $i\in\{1,2\}$. Let $c_i\in L_i$, for $i\in\{1,2\}$, be distinct colors. To obtain an
$s$-edge-coloring  of $G$ we first, by induction, $s$-edge-color $G[X_{1}\cup\{v\}]$, such that edges incident with $v_{1}$ receive colors from the list $L_{1}$ and edges incident with $v$ (in $G[X_{1}\cup\{v\}]$) colors from the list $\{1,2,\ldots,s\}\setminus\{c_2\}$. Let $L$ be the set of colors used for coloring edges incident with $v$ in this coloring, and let $L'=\{1,2,\ldots,s\}\setminus L$. Then, to obtain a desired edge-coloring of $G$ we, by induction, $s$-edge-color $G[X_{2}\cup\{v\}]$, so that edges incident with $v_{2}$ receive colors from the list $L_{2}$ and edges incident with $v$ (in $G[X_{2}\cup\{v\}]$) colors from the list $L'$ (note that $c_2\in L'\cap L_2$).

So, we may assume  w.l.o.g.\ that $v_1,v_2\in X_1\cup\{v\}$. To obtain an $s$-edge-coloring of $G$, we first, by induction, $s$-edge-color $G[X_1\cup \{v\}]$ so that the lists $L_1$ and $L_2$ are associated with $v_1$ and $v_2$. Then we $s$-edge-color $G[X_2\cup\{v\}]$ using Theorem \ref{ColorChordless} and permute the colors in this edge-coloring so that edges incident with $v$ (in $G$) all receive different colors.

\medskip
\noindent{\bf Case 4.} $G$ has a proper 2-cutset $\{a,b\}$.
\medskip

\noindent By Case 2 we may assume that $G$ is 2-connected. Let a proper 2-cutset $\{a,b\}$ of $G$, with the split $(X_1,X_2,\{ a,b\} )$, be chosen so that the side $X_1$ is
minimum among all such splits, that is, by Theorem \ref{DecomposChordless}, such that the block of decomposition $G_{X_1}$ is sparse and $a$ and $b$ each have at least
two neighbors in $X_1$.
Then, $G[X_1\cup\{a,b\}]$ is also triangle-free sparse, and since $G_{X_1}$ is sparse, each of the vertices $a$ and $b$ is of degree at least 3 or free in $G[X_1\cup\{a,b\}]$.
%Furthermore, if $a_1$ (resp.\ $b_1$) has only one neighbor in $X_2'$, then we define $P_1=a_1\ldots a_2$ (resp.\ $P_2=b_1\ldots b_2$) to be the limb of $G[X_2'\cup\{a_1,b_1\}]$ that contains $a_1$ (resp.\ $b_1$); otherwise $P_1=\{a_1\}$ (resp.\ $P_2=\{b_1\}$) and $a_2=a_1$ (resp.\ $b_2=b_1$). Let $X_2=X_2'\setminus(V(P_1)\cup V(P_2))$. Then the following holds.
\smallskip

\noindent{\bf Case 4.1.} $v_i\in X_i$, for $i\in\{1,2\}$.
\smallskip

\noindent First, let us assume that $v_1$ is adjacent to both $a$ and $b$.
Since $G_{X_1}$ is
sparse, $v_1$ is of degree 2. By the minimality of $X_1$, this
implies that $G[(X_1\setminus\{v_1\})\cup\{a,b\}]$ is a chordless
path. Let $Q'$ be a chordless path in $G$ whose one endnode is $v_2$, the other is a vertex of $\{ a,b\}$, and no interior vertex is in $\{ a,b\}$.
Then $V(Q)\cup V(Q')\cup X_1\cup\{a,b\}$ induces in $\widetilde{G}$ a cycle with a chord ($av_1$ or $bv_1$), a contradiction.
Therefore $v_1$ cannot be adjacent to both $a$ and $b$.
Similarly, $G[X_1\cup\{a,b\}]$ is not a hole of length 5. Indeed, if we suppose the opposite, then $v_1$ is adjacent to $a$ or $b$, w.l.o.g.\ to $a$. Now, if $Q'$ is a path from $v_2$ to $a$ in $G[X_2\cup \{ a,b\} ]$ whose interior does not go through $b$, then $V(Q)\cup (Q')\cup X_1\cup\{a,b\}$ induces a cycle with chord $av_1$, a contradiction. Finally, note that $\{a,b\}$ is not contained in a ring of length 5 of $G[X_1\cup\{a,b\}]$. Indeed, if we suppose the opposite, then, since $G[X_1\cup\{a,b\}]$ is not a hole of length 5, the degree at least 3 vertex of this ring is a cutvertex of $G$.

By the previous paragraph w.l.o.g.\ $v_1$ is not adjacent to $b$. Next, suppose that $v_1a$ is also not an edge. Then, to
obtain an $s$-edge-coloring of $G$ we first $s$-edge-color
$G[X_{2}\cup\{a,b\}]$ using Theorem \ref{ColorChordless}, and then permute the colors so that edges incident with $v_{2}$
receive colors from the list $L_{2}$. Let $L_a'$ (resp.\ $L_b'$)
be the set of colors used for coloring edges incident with $a$
(resp.\ $b$) in this coloring, and let
$L_a=\{1,2,\ldots,s\}\setminus L_a'$ (resp.\
$L_b=\{1,2,\ldots,s\}\setminus L_b'$). We
complete $s$-edge-coloring of $G$ using part (2) of Lemma \ref{2list2sparse}, that is,
we $s$-edge-color $G[X_1\cup\{a,b\}]$ so that edges incident with
$v_1$, $a$ and $b$ receive colors from the lists $L_1$, $L_a$ and
$L_b$.

Hence, we may assume that $v_1$
is adjacent to $a$ (but not to $b$). Let
$c_1\in L_1$, and  $Q'$ be the path induced by $V(Q)\cup \{v_1\}$. Then $Q'$ is disjoint from $G[X_2\cup\{a,b\}]$, its endnodes are adjacent to $a$ and $v_2$, and
$\widetilde{G}[X_2\cup\{a,b\}\cup V(Q')]$ is triangle-free chordless. So, to obtain an
$s$-edge-coloring of $G$ we first, by induction, $s$-edge-color
$G[X_{2}\cup\{a,b\}]$, so that edges incident with $v_{2}$
receive colors from the list $L_{2}$ and edges incident with $a$ (in
$G[X_{2}\cup\{a,b\}]$) colors from the list
$\{1,2,\ldots,s\}\setminus\{c_1\}$. Let $L_a'$ (resp.\ $L_b'$) be
the set of colors used for coloring edges incident with $a$ (resp.\
$b$) in this coloring, and let
$L_a=\{1,2,\ldots,s\}\setminus L_a'$ (resp.\
$L_b=\{1,2,\ldots,s\}\setminus L_b'$).  To complete $s$-edge-coloring of $G$ we $s$-edge-color
$G[X_{1}\cup\{a,b\}]$ using part (2) of Lemma \ref{2list2sparse}, so that
edges incident with $v_{1}$, $a$ and $b$ receive colors from the
lists $L_{1}$, $L_a$ and $L_b$, respectively (note that
$c_1\in L_1\cap L_a$).
\smallskip

\noindent{\bf Case 4.2.} $v_1,v_2\in X_i\cup\{a,b\}$, for some $i\in\{1,2\}$.
\smallskip

\noindent Note that if $a$ and $b$ are of degree 1 in $G[X_j\cup\{a,b\}]$, for some $j\in\{1,2\}$, then their neighbors in $X_j$ are distinct. Indeed, if we suppose the opposite, then their common neighbor is a cutvertex of $G$.

Now, to obtain an $s$-edge-coloring of $G$ we first, by induction, $s$-edge-color $G[X_{i}\cup\{a,b\}]$, so that edges incident with $v_{j}$, for $j\in\{1,2\}$, receive colors from the list $L_{j}$. Let $L_a$ (resp.\ $L_b$) be the set of colors used for coloring edges incident with $a$ (resp.\ $b$) in this coloring, and let $L_a'=\{1,2,\ldots,s\}\setminus L_a$ (resp.\ $L_b'=\{1,2,\ldots,s\}\setminus L_b$). Let $Q_{3-i}'$ be a chordless path from $a$ to $b$ in $G[X_i\cup\{a,b\}]$, and $Q_{3-i}$ be the path induced by $V(Q_{3-i}')\setminus\{a,b\}$. Then $Q_{3-i}$ is disjoint from $G[X_{3-i}\cup\{a,b\}]$, its endnodes are adjacent to $a$ and $b$, and $G[X_{3-i}\cup\{a,b\}\cup V(Q_{3-i})]$ is triangle-free chordless.
So, to complete $s$-edge-coloring of $G$ we, by induction, $s$-edge-color $G[X_{3-i}\cup\{a,b\}]$, so that edges incident with $a$ and $b$ receive colors from the lists $L_a'$ and $L_b'$.
\medskip

Finally, let us explain how this proof yields an $\mathcal O(n^3m)$-time algorithm. By Case 1, in linear time we can  reduce the problem to the one where $\delta(G)\geq 2$. By Lemma \ref{2list2sparse}, Case 2 can be done in $\mathcal O(nm)$ time. In Case 3, either we use induction and apply Lemma \ref{2list2sparse}, or we use Theorem \ref{ColorChordless} to color the entire side. In each step of Case 4 we first find a desired 2-cutset, which can be done in $\mathcal O(n^2m)$ time (see \cite{mft:chordless}), and then edge-color "the basic" side, which can be done in $\mathcal O(nm)$ time (by Lemma \ref{2list2sparse}). Since there is at most $\mathcal O(n)$ steps and each time we use Theorem \ref{ColorChordless} the entire side is colored, the running time of the algorithm is $\mathcal O(n\cdot(n^2m+nm)+n^3m)=\mathcal O(n^3m)$, as claimed.
\end{proof}

\begin{lemma}\label{vcn1}
Let $G \in \mathcal D$ and let $(X_1,X_2,A_1,A_2,B_1,B_2)$ be a split of a minimally-sided 2-join of $G$, with $X_1$ being
its minimal side, and let $G_1$ and $G_2$ be the corresponding blocks of decomposition. Let $s=\max\{3,\omega(G)\}$, and assume that we
are given an $s$-coloring $c$ of $G[X_2]$. We can extend $c$ to an $s$-coloring of $G$ in ${\cal O} (n^3m)$-time.
Furthermore, if $G$ is a basic graph then it can be $\max\{3,\omega(G)\}$-colored in ${\cal O} (n^3m)$-time.
\end{lemma}

\begin{proof}
By Lemma \ref{extreme}, $|A_1|,|B_1|\geq 2$, and by Lemma \ref{l:consistent}, $(X_1,X_2)$ is a consistent 2-join.
Hence $A_2$ and $B_2$ are cliques. Also, by Lemma \ref{new2}, $G_1\in \mathcal D$, and by Lemma \ref{extreme}, $G_1$ does not have a 2-join.
So, by Theorem \ref{decomposeTW}, $G_1$ is basic.

Let $L_a$ (resp. $L_b$) be the set of colors that $c$ assigns to vertices of $A_2$ (resp. $B_2$).
Let $L_a'=\{ 1, \ldots ,s\} \setminus L_a$ and $L_b'=\{ 1, \ldots ,s\} \setminus L_b$.
We want an $s$-coloring of $G[X_1]$ in which the vertices of $A_1$ are colored with colors from $L_a'$ and
 vertices of $B_1$ are colored with colors from $L_b'$.

 First suppose that $G_1$ is a line graph of a triangle-free chordless graph. Then $G_1$ is claw-free and hence (since $|A_1|,|B_1|\geq 2$) $A_1$ and $B_1$
 are both cliques. Let $R$ be the triangle-free chordless graph such that $L(R)=G[X_1]$.
 Since $R$ is triangle-free, $A_1$ (resp. $B_1$) corresponds to the set of edges incident to vertex $v_1$ (resp. $v_2$) of $R$ that is of degree at least 1 in $R$.
 Note that $v_1$ and $v_2$ are not adjacent since $A_1\cap B_1=\emptyset$.
 Since $A_1,A_2,B_1,B_2$ are all cliques, $\deg_R(v_1)\leq |L_a'|$ and $\deg_R(v_2)\leq |L_b'|$.
 We associate lists $L_a'$ and $L_b'$ to vertices $v_1$ and $v_2$ respectively, and the result follows from Lemma \ref{TwoColoredBasic}.

 Now suppose that $G_1$ is a P-graph with special clique $K$ and skeleton $R$. Let $K'$ be the vertices of $K$ that are centers of claws.
 Note that all centers of claws of $G_1$ are in $K'$. For $u\in K'$, by (viii) of the definition of the skeleton of a P-graph, all pendant vertices of $L(R)$ that are adjacent to $u$
 are of degree 2 in $G_1$. Let $H$ be the graph obtained from $G[X_1]$ by removing degree 2 vertices of $G_1$ that are adjacent to a vertex of $K'$.
 Then $H$ is claw-free, and hence by Lemma \ref{p1l2.4} and Lemma \ref{p2l4.2}, $H$ is the line graph of a triangle-free chordless graph, say $R_H$.

 If $A_1$ and $B_1$ are both cliques, then (since $|A_1|,|B_1|\geq 2$) $A_1\cup B_1\subseteq V(H)$, and we $s$-color $H$, similarly to the case
 when $G_1$ was the line graph of triangle-free chordless graph, so that vertices of $A_1$ (resp. $B_1$) are colored with colors from $L_a'$ (resp. $L_b'$).
 This coloring easily extends to an $s$-coloring of $G[X_1]$ since $s\geq 3$.

 So we may assume that $A_1$ is not a clique. Since $G\in \mathcal D$, by Lemma \ref{diamondCliqueCut}, $G$ is diamond-free, and hence (since $|A_1|\geq2$) it
 follows that $|A_2|=1$.
 Therefore $|L_a'|\geq 2$. Let $a_2$ be the vertex of the marker path of $G_1$ that is complete to $A_1$. Since $A_1$ is not a clique, $a_2$ is center of a claw and hence
 $a_2\in K'$. It follows that $K\cap X_1\subseteq A_1$, and so $B_1$ is a clique.
 Let $A_1'=A_1\cap V(H)$ and $A_1''=A_1\setminus A_1'$. So vertices of $A_1''$ are all of degree 2 in $G_1$, and hence of degree 1 in $H$.
 Also, $A_1'\subseteq K$ (since the vertices that are adjacent to centers of claws of $G_1$, and in particular to $a_2$, must be either in $K$ or of degree 2 in $G_1$),
 and hence $A_1'$ is a (possibly empty) clique.

 We first $s$-color $H$ so that the vertices of $A_1'$ (resp. $B_1$) receive the colors from $L_a'$ (resp. $L_b'$), and then we extend this coloring to the desired coloring
 of $G[X_1]$.
 Clique $B_1$ of $G_1$ corresponds to edges incident to a vertex $v_2$ of $R_H$. We assign list $L_b'$ to $v_2$. Since $B_1$ and $B_2$ are cliques,
 $\deg_{R_H} (v_2)\leq |L_b'|$.
 If $A_1'=\emptyset$ then we $s$-color $H$ by Theorem \ref{ColorChordless} (in ${\cal O} (n^3m)$-time) and then permute colors so that the vertices of $B_1$ are
 colored with colors from $L_b'$. So let us assume that $A_1'\neq \emptyset$, and let $v_1$ be the vertex of $R_H$ whose incident edges correspond to vertices of
 $A_1'$. We assign list $L_a'$ to $v_1$. Since $A_2$ and $A_1'$ are cliques, $\deg_{R_H} (v_1)\leq |L_a'|$. Note that $v_1$ and $v_2$ are not adjacent in $R_H$ since
 $A_1\cap B_1=\emptyset$. It now follows from Lemma \ref{TwoColoredBasic} that we can obtain the desired $s$-coloring of $H$ in ${\cal O} (n^3m)$-time.
 So, we may assume that we have an $s$-coloring of $H$ in which vertices of $A_1'$ (resp. $B_1$) are colored with colors from $L_a'$ (resp. $L_b'$).
 We now extend that to the desired $s$-coloring of $G[X_1]$. Since $s\geq 3$ and vertices of $X_1\setminus H$ all have degree 2, we can greedily extend the coloring of $H$ to
 vertices of $X_1\setminus (H\cup A_1'')$. Since $|A_2|=1$ and $s\geq 3$, it follows that $|L_a'|\geq 2$. Since vertices of $A_1''$ are of degree 1 in $H$, we can clearly
 extend the coloring to them as well, so that they receive a color from $L_a'$.

 Therefore, $s$-coloring of $G[X_2]$ can be extended to an $s$-coloring of $G$ in ${\cal O} (n^3m)$-time. Observe that this proof also shows that any basic graph can be colored
 in ${\cal O} (n^3m)$-time.
\end{proof}

\begin{theorem}%\label{ColoringAlg}
There is an algorithm with the following specifications:
\begin{description}
\item[ Input:] A graph $G\in\mathcal C$.
\item[ Output:]
A $\chi(G)$-coloring of $G$.
\item[ Running time:] $\mathcal O(n^5m)$.
\end{description}
Furthermore, if $G\in {\cal C}$  then $\chi(G)\leq \max\{3,\omega(G)\}$.
\end{theorem}

\begin{proof}
We can decide in linear time if $G$ is 2-colorable, and if it is 2-colorable we can 2-color it (also in linear time). So it is enough to give an algorithm that outputs a
$\max\{3,\omega(G)\}$-coloring of $G$.

\medskip
\noindent
{\bf Claim:} {\em Every $G\in \mathcal D$ can be $\max\{3,\omega(G)\}$-colored in $\mathcal O(n^4m)$-time.}

\medskip
\noindent
{\em Proof of Claim:} Let $G\in \mathcal D$ and let $s=\max\{3,\omega(G)\}$. We $s$-color $G$ as follows.
First check whether $G$ contains a 2-join (this can be done in $\mathcal O(n^2m)$-time by the algorithm in \cite{fast2j}). If it does not, then by Theorem \ref{decomposeTW}
$G$ is basic, and hence it can be $s$-colored in $\mathcal O(n^3m)$-time by Lemma \ref{vcn1}.
Otherwise, by Lemma \ref{DT-construction}, we construct a 2-join decomposition tree $T_G$ (of depths $1\leq p\leq n$) using marker paths of length 3, in
$\mathcal O(n^4m)$-time.
By Lemma \ref{new3} all graphs $G_B^1, \ldots ,G_B^p,G^p$ that correspond to the leaves of $T_G$ are in $\mathcal D_{\textsc{basic}}$.

All the 2-joins used in the construction of $T_G$ are extreme 2-joins. For our purpose here we want them to be minimally-sided 2-joins.
Note that by Lemma \ref{extreme}, every minimally-sided 2-join is an extreme 2-join, but not every extreme 2-join is a minimally-sided one.
The way $T_G$ is constructed in \cite{nicolas.kristina:two} first a minimally-sided 2-join is found and then in order to achieve ${\cal M}$-independence,
it is possibly pulled in the direction of minimal side to obtain another extreme 2-join that is then used in the construction of $T_G$.
If we do not care about ${\cal M}$-independence (as we do not here), we can have the algorithm that constructs $T_G$ just use the minimally-sided 2-join that is first
found. This way we obtain $T_G$ with all the other properties, except ${\cal M}$-independence, in which every 2-join used is minimally-sided (which is what we need here).

To obtain the desired coloring of $G$, we process vertices of $T_G$ from bottom up. We start with $G^p$. As $G^p$ is basic, we color it in $\mathcal O(n^3m)$-time
by Lemma \ref{vcn1}.
Since $G^p$ and $G^p_B$ are blocks of decomposition w.r.t.\ a minimally-sided 2-join of $G^{p-1}$, with $G^p_B$ being a block that corresponds to a minimal side,
by Lemma \ref{vcn1} we extend the coloring of $G^p$ to $G^{p-1}$ in $\mathcal O(n^3m)$-time.
We proceed like this up the tree, all the way to the root of $T_G$, namely $G^0=G$. As the depth of $T_G$ is at most $n$, it follows that $G$ can be $s$-colored
in $\mathcal O(n^4m)$-time.
This completes the proof of the Claim.

\medskip
We now consider $G\in {\cal C}$. By Theorem \ref{th:tarjan} we construct the clique-cutset decomposition tree $T$ of $G$ in $\mathcal O(nm)$-time.
So all the leaves of $T$ are graphs from $\mathcal D$, and there are at most $n$ of them. So to $s$-color all the leaves, by the Claim, it takes time
$\mathcal O(n\cdot n^4m)=\mathcal O(n^5m)$.
Finally, process the tree from bottom up, permuting colors of the blocks of decomposition so they agree on the clique cutset and paste the colorings of the
blocks together. Going this way all the way up to the root of $T$, we obtain the desired coloring of $G$ in $\mathcal O(n^5m)$-time.
\end{proof}

\section{A note on clique-width}
\label{s:cW}

%In previous sections we have shown that many problems that are NP-hard
%in general are polynomial in $\mathcal C$. So, having in mind
%\cite{courcelle:CW}, one might think that graphs in $\mathcal C$ have
%bounded clique-width (or, equivalently, bounded rank-width -- see
%\cite{oum}). In this short section we present a subclass of
%$\mathcal C$, due to Lozin and Rauthenbach (see
%\cite{lozin:unboundedCW}), of unbounded clique-width (in fact, we prove that $\mathcal D_{\textsc{basic}}$ has unbounded clique-width).

In this section we show that the class $\mathcal D_{\textsc{basic}}$ has unbounded clique-width (and hence unbounded rank-width \cite{oum}).
So the class of (theta,wheel)-free graphs with no clique cutset has unbounded clique-width.

For $k\geq 3$, let $C_k$ be a chordless cycle of length $k$. For
$k\geq 1$, let $H_k$ be the graph on vertex set
$\{w_1,\ldots,w_{k+1},u',u'',v',v''\}$, such that
$\{w_1,\ldots,w_{k+1}\}$ induces a chordless path of length $k$, and
the only other edges of $H_k$ are $u'w_1$, $u''w_1$, $v'w_{k+1}$
and $v''w_{k+1}$.

Let  $\Phi_k$ be the class of planar bipartite
$(C_3,\ldots,C_k,H_1,\ldots,H_k)$-free graphs of vertex degree at
most~3.

\begin{lemma}[\cite{lozin:unboundedCW}]
  For any positive integer $k$, the tree- and clique-width of graphs
  in $\Phi_k$ is unbounded.
\end{lemma}

Note that every $(H_1,C_3,C_4)$-free graph is chordless and triangle-free, so the class of triangle-free chordless graphs is the superclass of $\Phi_4$, which, by previous lemma, has unbounded clique-width. Furthermore, if $\Phi_4'$ is the class of 2-connected graphs from $\Phi_4$, then $\Phi_4'$ also has unbounded clique-width (see, for example, \cite{klm}). Moreover, the following holds.

\begin{lemma}[\cite{klm}]
  If $\mathcal G$ is a class of graph that has unbounded clique-width, then the class $L(\mathcal G):=\{L(G)\,:\,G\in\mathcal G\}$ also has unbounded clique-width.
\end{lemma}

This lemma, together with our previous observations, implies that the class $L(\Phi_4')$ has unbounded clique-width. Since $L(\Phi_4')\subseteq \mathcal D_{\textsc{basic}}$, we conclude that the class $\mathcal D_{\textsc{basic}}$  has unbounded clique-width.

Interestingly, N.K. Le \cite{le} proved that the class of (theta, wheel, prism)-free graphs that do not have a clique cutset has bounded clique-width
(using the decomposition theorem for this class from \cite{twf-p1}).
This means that  one can use the machinery of
 \cite{courcelle:CW} and \cite{tarjan} to obtain faster polynomial-time algorithms for coloring and stable set problem for (theta, wheel, prism)-free graphs.


\begin{thebibliography}{99}

%
%\bibitem{bienstock}
%D. Bienstock.
%\newblock On the complexity of testing for even holes and induced paths.
%\newblock{\em Discrete Mathematics}, 90:85--92, 1991.
%\newblock See also Corrigendum by B. Reed, {\em Discrete Mathematics},
%102:109, 1992.

\bibitem{fast2j}
P. Charbit, M. Habib, N. Trotignon, K. Vu\v{s}kovi\'c.
\newblock Detecting 2-joins faster.
\newblock {\em Journal of Discrete Algorithms}, 17: 60-66, 2012.

\bibitem{courcelle:CW}
B. Courcelle, J.A. Makowsky, U. Rotics,
\newblock Linear time solvable optimization problems on graphs on bounded clique width.
\newblock {\em Theory of Computing Systems}, 33 (2): 125-150, 2000.




\bibitem{twf-p1}
E. Diot, M. Radovanovi\'c, N. Trotignon, K. Vu\v{s}kovi\'c.
\newblock The (theta,wheel)-free graphs Part I: ony-prism and only pyramid graphs. arXiv:1504.01862


\bibitem{edmonds}
J. Edmonds.
\newblock Paths, trees and flowers.
\newblock {\em Canad.\ J.\ Math.}, 17: 449-467, 1965.


\bibitem{klm}
M. Kami\'nski, V. Lozin, M. Milani\v c.
\newblock Recent developments on graphs of bounded clique width.
\newblock {\em Discrete Applied Mathematics}, 157: 2747--2761, 2009.


\bibitem{le}
N.K. Le, private communication.



\bibitem{lmt:isk4}
B.~L{\'e}v{\^e}que, F.~Maffray,  N.~Trotignon.
\newblock On graphs with no induced subdivision of {$K_4$}.
\newblock {\em Journal of Combinatorial Theory, Series B}, 102: 924--947, 2010.


\bibitem{lozin:unboundedCW}
V. Lozin, D. Rauthenbach.
\newblock The tree- and clique-width of bipartite graphs in special classes.
\newblock {\em Australasian Journal of Combinatorics}, 34: 57-67, 2006.

%
\bibitem{mft:chordless}
R.C.S. Machado, C.M.H. de~Figueiredo, N.~Trotignon.
\newblock Edge-colouring and total-colouring chordless graphs.
\newblock {\em Discrete Mathematics},  313 (4): 1547-1552, 2010.


\bibitem{oum}
S. Oum, P. Seymour.
\newblock Approximating clique-width and branch-width.
\newblock {\em Journal of Combinatorial Theory, Series B}, 96 (4): 514-528, 2006.
%


\bibitem{twf-p2}
M. Radovanovi\'c, N. Trotignon, K. Vu\v{s}kovi\'c.
\newblock The (theta,wheel)-free graphs Part II: structure theorem. arXiv:1703.08675



\bibitem{twf-p4}
M. Radovanovi\'c, N. Trotignon, K. Vu\v{s}kovi\'c.
\newblock The (theta,wheel)-free graphs Part IV: induced cycles and paths.


%\bibitem{RS:GraphMinor13}
%N. Robertson and P.D. Seymour.
%\newblock Graph minors. XIII. The disjoint paths problem. \newblock {\em Journal of Combinatorial Theory, Series B}, 63 (1995), 65--110.



\bibitem{tarjan}
R.E. Tarjan.
\newblock Decomposition by clique separators.
\newblock {\em Discrete Mathematics}, 55: 221-232, 1985.




\bibitem{nicolas.kristina:two}
N.~Trotignon, K.~Vu{\v s}kovi{\'c}.
\newblock Combinatorial optimization with 2-joins.
\newblock {\em Journal of Combinatorial Theory, Series B}, 102:153-185, 2012.


\end{thebibliography}
\end{document}